\documentclass[12pt,a4paper,oneside]{article}
\usepackage{amsmath,amstext,amssymb,amsopn,amsthm,color}
\usepackage{graphicx}

\theoremstyle{plain}
 \newtheorem{thm}{Theorem}[section]
\newtheorem{lem}[thm]{Lemma} \newtheorem{prop}[thm]{Proposition}
\newtheorem{cor}[thm]{Corollary}
\theoremstyle{definition}
\newtheorem{exmp}{Example} \theoremstyle{remark}
\newtheorem*{rem*}{Remark} \newtheorem{rem}{Remark}
 \newcommand{\R}{\mathbb{R}}
 \newcommand{\Rd}{{\R^{d}}}

\renewcommand{\leq}{\leqslant} \renewcommand{\le}{\leq}
\renewcommand{\geq}{\geqslant} \renewcommand{\ge}{\geq}

\def\Go{G_{\{0\}^c}}
   
 \def\({\left(} \def\){\right)} \def\[{\left[}
  \def\]{\right]} \def\<{\langle} \def\>{\rangle} 
   
\def\E{\mathbb{E}}
\def\p{\mathbb{P}}

\definecolor{mr}{rgb}{0.1,0.2,0.7}
\definecolor{tg}{rgb}{0.7,0.1,0.2}

\newcommand{\WUSC}[2]{\textrm{\rm WUSC}(#1,#2)}
\newcommand{\WLSC}[2]{\textrm{\rm WLSC}(#1,#2)}

\newcommand{\Ca}{C_{9}}
\newcommand{\Cb}{C_{8}}
\newcommand{\Cc}{C_{7}}
\newcommand{\Cd}{C_{6}}
\newcommand{\Ce}{C_{5}}

\title{Hitting times of points and intervals for symmetric L\'{e}vy processes \thanks{
exponent, hitting time, Harnack inequality, Dirichlet kernel}}
\author{Tomasz Grzywny\thanks{Wroc\l{}aw University of Technology, ul. Wyb. Wyspia\'{n}skiego 27, 50-370 Wroc\l{}aw, Poland,
tomasz.grzywny@pwr.edu.pl, michal.ryznar@pwr.edu.pl} and Micha\l{} Ryznar\textsuperscript{\textdagger}\\
  Wroc\l{}aw University of Technology, Poland}

\begin{document}
\maketitle
\begin{abstract}
For  one-dimensional symmetric L\'evy processes, which hit every  point with positive probability,   we give
sharp bounds for the tail function $\p^x(T_B>t)$, where $T_B$ is the first hitting time of $B$ which is either a single point or an  interval.
The estimates are obtained under some weak type scaling  assumptions on the characteristic exponent of the process. We apply these results to prove optimal estimates of the transition density of the process killed after hitting $B$. \end{abstract}

\section{Introduction}

\setcounter{equation}{0}

The purpose of this paper is to investigate the distribution of the first hitting time of a point or an  interval by a symmetric L\'evy process such that $\{0\}$ is regular for itself. Such processes hit points with positive probability. Our main results, under certain regularity assumptions,  provide sharp estimates of the tail function $\p^x(T_0>t), t>0$, where  $T_0$ is the first hitting time of the point $0$ by the process starting from $x$. We further derive similar estimates for the first hitting time of an interval of a given width, under some weak scaling assumption on the characteristic exponent $\psi$ of the process. We also find the asymptotic behaviour of the tail function either for the first hitting time of a point or a compact set  under the assumption that the characteristic exponent is regularly varying  at zero with index $\delta\ge 1$. The estimates or asymptotics obtained in the paper are expressed in terms of the generalized inverse  $\psi^{-1}$ of the characteristic exponent and  the compensated potential kernel $$K(x)=\int^\infty_0(p_s(0)-p_s(x))ds,\ x\in \R.$$
Here $p_s(x), s>0, x\in \R,$ is the transition density of the process, which must exist for processes we study. If $\psi$ is comparable with a non-decreasing function  we are able to provide sharp estimates of $K$ in terms of the characteristic exponent, so in these cases the estimates become quite explicit and given in terms of the characteristic exponent and its generalized inverse.  For example we show that if $\psi$ has the weak lower scaling property with index $\alpha>1$ (see Preliminaries for the definition) then
$$ \p^x(T_0>t)\approx  \frac{1}{t\psi^{-1}(1/t)|x|\psi(1/x)}\wedge 1, \quad x\in \R,\, t>0.$$
Moreover, we find a similar estimate in the case when $ \p^x(T_0>t)$ is replaced by the tail function of the first hitting time of an interval (see Theorem \ref{ball}). While  in principle, for starting points $x$ far away from the interval, such estimates should follow from the estimates of $ \p^x(T_0>t)$, but for points close to the boundary of the interval the behaviour of the tail function is not clear.  In order to overcome this difficulty we proved and then applied  the global Harnack inequality under the weak scaling assumption for $\psi$ (see Theorem \ref{Harnack}). The  Harnack inequality is one of the central topics in the potential theory and the present paper contributes to these studies. Usually the Harnack inequality for L\'evy or generally Markov processes is proved under the assumptions which enforce the transience of the process and absolute continuity of its L\'evy measure \cite{MR1918242,MR2031452,MR2524930,MR3225805,MR3271268}. In our case the process is not only recurrent but point recurrent.

Finally, under the assumption the process is  unimodal and $\psi$ has the  lower and upper weak  scaling property we apply the estimates   of the hitting times and derive sharp estimates of $p^{D}$, the transition density (heat kernel)  of the process killed after hitting an interval.  We show that for $D=(-\infty, -R) \cup (R, \infty),\ R>0$ we have
$$p_t^{D}( x, y)\approx   \p^x(\tau_D>t)\p^y(\tau_D>t) p_t(x-y),\quad \ t>0,\,\,\, x,\,y \in D, $$
with the comparability constant independent of $R$ and where $\tau_D$ denotes the first exit time from $D$.
The problem of estimating the heat kernel for symmetric L\'evy  processes has brought a lot of attention recently; see e.g. \cite{MR2677618,MR2722789,MR2923420,MR3131293,MR3249349}.
Most of the results are derived under the assumption that the process is transient. The recurrent processes, except isotropic stable \cite{MR2722789}, were not explored with regard to heat kernel estimates for exterior sets, and to the best of our knowledge our result is the first one with such generality. One of the drawbacks is that we deal with one-dimensional processes which are point recurrent. It would be desirable to provide such optimal estimates for one or two-dimensional recurrent symmetric L\'evy  processes, which do not hit points. Unfortunately our approach, based on the nice  behaviour of the compensated kernel $K$, will not work in this case.

The distribution of the hitting time of points or compact sets for one-dimensional $\alpha$-stable processes was a subject of studies in several papers \cite{MR0217877,MR1149016,MR2599211,MR2800088,MR2466193,MR3183574,MR2988398,2014arXiv1403.3714J}.
Let $T_B$ be the first hitting time of a set $B$.  Port in  \cite {MR0217877} found
 the asymptotics  of $\p^x(T_B>t), t\to \infty$ for a compact  set $B$ if  $1< \alpha<2$, and for not necessarily symmetric stable processes. The density $f_x(t)$ of $T_{x}$ for the symmetric $\alpha$-stable process, $1<\alpha<2$,   was found in \cite {MR2599211}.
  For spectrally positive (no negative   jumps) $\alpha$-stable process, $1<\alpha<2$,  Peskir \cite{MR2466193}
and Simon \cite{MR2800088}
found the density $f_x(t), x>0,$  in a form of a series from which one can derive the asymptotics   of $f_x(t)$ as $t\to 0^+$  or $t\to\infty$. In a recent paper
\cite{MR3183574}
 this type of result was extended to    $\alpha$-stable processes, $1<\alpha<2$, having both negative and positive jumps. In this paper the authors derived the Mellin transform of the distribution of  $T_{x}$ and then successfully inverted it to obtain the series representation of the density of $T_{x}$.

Relatively little is known about the distribution of  hitting times of single points
for general L\'evy processes. To our best knowledge  such explicit results as mentioned above do not exist. Only recently Kwa\'snicki \cite{MR2988398}
studied the distribution of $T_{x}$ for symmetric L\'evy processes under certain regularity assumptions on the characteristic exponent of the process. The main result provides an integral  representation of the distribution function of $T_{x}$ in terms of generalized eigenfunctions for the killed semigroup upon hitting $\{0\}$. This representation was then successfully applied in \cite{2014arXiv1403.3714J} to obtain various asymptotics and estimates  of the tail function of $T_{x}$ and its derivatives under further additional regularity assumptions on  characteristic exponent and  the L\'evy measure. Namely it is assumed that the process has  completely monotone L\'evy density.
Comparing our results with those obtained in \cite{2014arXiv1403.3714J} we remark that our assumptions are much less restrictive, however our approach  does not allow to treat the estimates of the density or the higher derivatives of the distribution functions. In a forthcoming paper we provide sharp  estimates of the density under the weak upper and lower scaling property for the characteristic exponent for   unimodal L\'evy  processes. Moreover we also treat the hitting distribution  of intervals and  provide sharp estimates and asymptotics of the tail function, which was not investigated in  \cite{2014arXiv1403.3714J}. We also mention that our methods are more elementary and are based on the estimates of the Laplace transforms of the hitting distributions and various estimates of exit probabilities.

{The paper is composed as follows. In Section 2 we recall some basic material regarding one-dimensional symmetric L\'{e}vy processes and present some auxiliary results which we use in the sequel.  In Section 3 we obtain estimates and asymptotics of the tail function $\p^x(T_0>t)$. Section 4 is devoted to the  uniform Harnack inequality and boundary behaviour of harmonic functions. These tools we use in Section 5 to prove estimates of the function $\p^x(T_{[-r,r]}>t)$.  Section 6 focuses on symmetric unimodal processes with weak global scaling. We use the methods and results of the previous sections to obtain estimates of the Dirichlet heat kernel of a complement of an interval.
}
\section{Preliminaries}

    Throughout the paper  by $c, c_1\,\dots$ we  denote
       nonnegative constants which may depend on other constant parameters only.
        The value of $c$ or $c_1\,\ldots$ may change from line to line in a chain
         of estimates. If we use enumerated  $C_1,C_2\, \dots$ then they are fixed constants and usually used in the sequel parts of the paper.
                Any subsets and real functions considered in the paper are assumed to be Borel measurable.
      The notion $p(u)\approx q(u),\ u \in A$ means that the ratio
      $p(u)/ q(u),\ u \in A$ is bounded from  below and above
      by positive (comparability) constants which may depend on other constant parameters only but does not depend on the set $A$.

We present in this section some basic material regarding one-dimensional   symmetric
L\'{e}vy processes which hit points with non-zero probability. For more  detailed
information, see \cite{MR1406564,MR0368175}. For questions regarding the
Markov and the strong Markov properties, semigroup properties,
Schr\"{o}dinger operators and basic potential theory, the reader
is referred to \cite{MR1329992} and \cite{MR0264757}.

In this paper we assume that a L\'evy process $X=(X_t, \,t\ge 0)$ \cite{MR1739520}, is symmetric. By $\nu$  we denote its L\'evy measure and by $\psi$  its    L\'evy-Khintchine exponent (symbol). Notice that $\nu$ and $\psi$ are symmetric as well.  Recall that any L\'evy measure is a measure such that
\begin{equation*}\label{wml}
\int_\R \(|x|^2\wedge 1\)\nu(dx)<\infty.
\end{equation*}
If the L\'evy measure $\nu$ is absolutely continuous with respect to the Lebesgue measure, then with a slight abuse of notation, we denote its density by $\nu$ as well.
Since the process is symmetric there is $\sigma\in \R$ such that 
$$
 \psi(\xi)=\int_\R \left(1- \cos \xi x\right) \nu(dx)+ \sigma^2 \xi^2 ,\qquad\xi\in\R,$$
and
$$
\E\,e^{i\xi X_t}=e^{-t\psi(\xi)},\qquad\xi\in\R.
$$ 

For $x\in \R$, by $\p^x$ and $\E^x$ we denote the distribution and the resulting expectation  of the process $x+X$. Obviously $\p^0=\p$ and $\E^0=\E$.
The process $X$ is called unimodal if for any $t>0$ the distribution $p_t(dx)$ of $X_t$ is  unimodal that is it is  absolutely continuous on $\R\setminus \{0\}$ and its density $p_t(x)$ is symmetric on $\R$ and {non-increasing} on $(0,\infty)$. Unimodal  L\'evy processes are characterized in \cite{MR705619} by unimodal L\'evy measures
$\nu(dx) = \nu(x) dx = \nu(|x|) dx$.

 The {\it first exit time} of an (open)
   set  $D\subset {\Rd}$
   by the process $X_t$ is defined by the formula
   $$
   \tau_{D}=\inf\{t> 0;\, X_t\notin D\}\,.
   $$

   If $F\subset \R$ is a closed set we define the first hitting time $T_F $ of $F$ as the first exit time from $F^c$. In the case when $F= \{a\}, a\in \R$ we denote  $T_F=T_a $. In this paper we consider symmetric L\'evy processes which have the property that  $0$ is regular for the set  $\{0\}$ that is
   $$\p^0(T_0=0)=1,$$
   which is equivalent to (\cite[Theoreme 7 and Theoreme 8]{MR0368175})
   \begin{equation}\label{0reg}\int^\infty_0\frac{1}{1+\psi(x)}dx<\infty.  \end{equation}
Note that the above condition implies that $\psi$ is unbounded, so excludes compound Poisson processes and in consequence $\psi(x)>0$ for $x\neq 0$. Moreover \eqref{0reg} guarantees that the distribution of $X_t, t>0$, is absolutely continuous and its density $p_t(\cdot)\in {C^\infty(\R)}$.


 In general potential theory  a very important role is played by $\lambda$-potential  kernels, $\lambda>0$  which are
defined as

$$u^\lambda(x)=\int_0^\infty e^{-\lambda t}p_t(x)dt, \ x\in \R.$$
If the defining integral above is finite   for $\lambda=0$ we call $u^0(x)=u(x)$ the potential kernel and then     the underlying process is transient.

Under the above assumptions it follows from \cite[Corollary II.18 and Theorem II.19]{MR1406564} that   $h^\lambda(x)=\E^0e^{-\lambda T_x}$ is continuous and $$u^\lambda(x)=\int^\infty_0 e^{-\lambda s}p_s(x)ds=u^\lambda(0)h^\lambda(x).$$
Denote
 $$K^\lambda(x)=u^\lambda(0)-u^\lambda(x)$$ and $$K(x)=\lim_{\lambda\to 0^+}K^\lambda(x)=\int^\infty_0(p_s(0)-p_s(x))ds.$$ 
By symmetry and \cite[Theorem II.19]{MR1406564},
$$K^\lambda(x)=\frac{1}{\pi}\int^\infty_0(1-\cos xs)\frac{1}{\lambda+\psi(s)}ds.$$
The monotone convergence theorem implies
$$K(x)=\frac{1}{\pi}\int^\infty_0(1-\cos xs)\frac{1}{\psi(s)}ds=\frac{1}{x\pi}\int^\infty_0(1-\cos s)\frac{1}{\psi(s/x)}ds.$$

For  a number of   results below  we  make the assumption that $K $ is non-decreasing. We do not know any general criterion which guarantees monotonicity, but it is clear that  sufficient conditions are:   $\psi(x)/x$ is  non-decreasing on $(0,\infty)$ or the process $X$ is unimodal.   Another interesting  problem is the question if monotonicity of $K$ implies some monotonicity properties of $\psi$.

Define $\kappa\in[0,\infty)$ by $$\kappa=\frac{\pi}{\int_{0}^\infty\frac{ds}{\psi(s)}}.$$

\begin{lem}[\cite{MR2603019}, Theorem 3.1]\label{php1} We have $\p^0(T_x=\infty)=\kappa K(x)$.  If  $\int^\infty_0 \frac{1}{\psi(s)}ds=\infty$ then for any $x\in \R$,  $\p^0(T_x<\infty)=1$.
\end{lem}

If    $\int^\infty_0 \frac{1}{\psi(s)}ds<\infty$ the process $X$ is transient and it is clear from Lemma \ref{php1} that  its potential kernel satisfies
$$u(x)=\int^\infty_0 p_s(x)ds=\frac{1}{\kappa} \p^x(T_0<\infty).$$


\begin{prop}
$K$ is subadditive on $\R$.
\end{prop}
\begin{proof}
Observe that $T_{x+y}\leq T_x+T_{x+y}\circ T_x$, where $\circ$ denotes the usual shift operation.  By the strong Markov property, for $\lambda>0$,
$$h^\lambda(x+y)\geq h^{\lambda}(x)h^\lambda(y), \quad x,\,y\in \R.$$
Hence
\begin{eqnarray*}
K(x)+K(y)-K(x+y)&=&\lim_{\lambda\to 0}\[(u^\lambda(0)-u^{\lambda}(x))+(u^\lambda(0)-u^\lambda(y))-(u^\lambda(0)-u^\lambda(x+y))\]\\
&=&\lim_{\lambda\to 0}u^\lambda(0)\[1-h^\lambda(x)-h^\lambda(y)+h^\lambda(x+y)\]\\
&\geq&\lim_{\lambda\to 0}u^\lambda(0)\[1-h^\lambda(x)-h^\lambda(y)+h^\lambda(x)h^\lambda(y)\]\\
&=&\lim_{\lambda\to 0}u^\lambda(0)[1-h^\lambda(x)][1-h^\lambda(y)]\geq 0.
\end{eqnarray*}
\end{proof}


The fundamental object of the potential theory is the {\it killed process} $X_t^D$
  when exiting the set $D$. It is defined in terms of sample paths up to time $\tau_D$.
  More precisely,
  $$
 \E^x f(X_t^D) =  \E^x[t<\tau_D; f(X_t)]\,,\quad t>0\,.
  $$
  The density function of the transition probability of the process $X_t^D$ is denoted
  by $p_t^{D}$. We have
  \begin{equation*}
  p_t^{D}(x,y) = p_t(y-x) -
   \E^x[t> \tau_D; p_{ t-\tau_D}(y-X_{\tau_D})]    \,, \quad x,\, y \in {\Rd}\,.\label{density100}
  \end{equation*}
  Obviously, we obtain
   $$
     p_t^{D}(x,y) \le p_t(y-x) \,, \quad x,\, y \in {\Rd}\,.
   $$
  $(p_t^{D})_{t>0}$ is a strongly contractive semigroup (under composition) and shares most
  of properties of the semigroup $ p_t$. In particular, it is strongly Feller and
  symmetric: $ p_t^{D}(x,y) =  p_t^{D}(y,x)$.
 The $\lambda$-potential measure  of the process $X_t^D$ started from $x$ is a Borel measure defined as
  $$
   G^\lambda_D(x,A)= \int_0^{\infty}e^{-\lambda t} \p^x(X_t^D\in A) \,dt\,,
   $$
 for any Borel subset $A$ of $\R$.
 For the L\'evy processes explored in the paper their  potential measures are absolutely continuous and the corresponding density is
  $\lambda$-potential kernel of the process $X_t^D$ and is called
  {\it  $\lambda$-Green function} of the set $D$. It is denoted by $G^\lambda_D$ and  we have
  $$
   G^\lambda_D(x,y)= \int_0^{\infty}e^{-\lambda t} p_t^{D}(x,y)\,dt\,.
  $$
If $\lambda=0$ the corresponding {\it  $0$-Green function} we simply call the {\it  Green function} of $D$ and denote  $G_D(x,y)$.

  Another important  object in the potential theory of $X_t$ is the
  {\it harmonic measure}  of the
  set $D$. It is defined by the formula:
  $$ 
  P_D(x,A)=
  \E^x[\tau_D<\infty; {\bf{1}}_A(X_{\tau_D})], \quad A\subset \R.
  $$
  The density kernel (with respect to the Lebesgue measure) of  the measure $P_D(x,A)$ (if it exists) is called the
  {\it Poisson kernel} of the set $D$.
  The relationship between the Green function of $D$ and the harmonic measure is provided by the Ikeda-Watanabe formula \cite{MR0142153},

  $$ 
  P_D(x,A)=  \int_D \nu(A-y)G_D(x,dy), \quad A\subset (\bar{D})^c.
  $$

Now we define harmonic and regular harmonic functions with respect to the process $X$. Let $u$ be
a Borel measurable function on $\R$. We say that $u$ is {\em
harmonic} function in an open set $D\subset \R$ if
$$u(x)=\E^xu(X_{\tau_B}), \quad x\in B,$$
for every bounded open set $B$ with the closure
$\overline{B}\subset D$. We say that $u$ is {\em regular harmonic} in $D$
if
$$u(x)=\E^x[\tau_D<\infty; u(X_{\tau_D})], \quad x\in D.$$
We note that for any open $D$ the Green function $G_D(x,y)$(if exists) is harmonic in $D\setminus \{y\}$  as a function of $x$. This follows from the strong Markov property and is frequently used in the paper.

 The following formula for the Green function of the complement of a point  can be found  in \cite[Lemma 4.1]{MR2603019},  \cite[Theorem 6.1]{MR1813843} for recurrent processes and \cite[Lemma 4]{MR2256481} for stable processes.
\begin{prop}$G_{\{0\}^c}(\cdot,\cdot)$ is jointly continuous on $\R\times \R$.
Moreover
\begin{equation}\label{G0}G_{\{0\}^c}(x,y)=K(x)+K(y)-K(y-x)-K(x)K(y)\kappa.\end{equation}
\end{prop}
\begin{proof}
For $\lambda>0$ we define $K^\lambda(x)=u^\lambda(0)-u^\lambda(x)$. We have
\begin{eqnarray*}\Go^\lambda(x,y)&=&u^\lambda(y-x)-\E^xe^{-\lambda T_{0}}u^\lambda(y-T_0)=u^\lambda(y-x)-h^\lambda(x)u^\lambda(y)\\
&=&-K^\lambda(y-x)+K^\lambda(y) +K^\lambda(x)-\frac{K^\lambda(x)K^\lambda(y)}{u^\lambda(0)}.
\end{eqnarray*}
Hence by the monotone convergence theorem
$$\Go(x,y)=K(x)+K(y)-K(y-x)-K(x)K(y)\kappa.$$
By the dominated convergence theorem we get continuity of $K$ and  $\Go$  as well.

\end{proof}

The following observation plays a crucial role in the sequel.

\begin{prop}\label{Greenb} For any $x,y \in \R$
we have $$\Go(x,y)\leq 2\[K(x)\wedge K(y)\].$$ If additionally  $K(\cdot)$ is non-decreasing then for  $xy\ge 0$ we have
$$\Go(x,y)\ge K(|x|\wedge|y|)-K(x)K(y)\kappa= K(|x|\wedge|y|)\p^{|x|\vee|y|}(T_0<\infty). $$

\end{prop}
\begin{proof}
By subadditivity of $K$ we have
$$K(y)\leq K(x)+K(y-x).$$\
Hence $$\Go(x,y)\leq 2K(x).$$
 If  $K(\cdot)$ is non-decreasing then for  $y\geq x>0$ we have  $K(y)-K(y-x)\geq 0$. Hence
$$K(x)-K(x)K(y)\kappa\le \Go(x,y).$$

\end{proof}

\begin{lem}\label{GreenRegHarm}
For any $0<|x|<R<|y|$,
$$\Go(x,y)= \E^x\Go(X_{\tau_{(-R,R)}\wedge T_0},y).$$
\end{lem}
\begin{proof}
Let $0<r<|x|<R<R^\prime<|y|$. Then by harmonicity of $\Go(\cdot,y)$ on $(-R^\prime,0)\cup (0,R^\prime)$ we have
$$\Go(x,y)=\E^x\Go(X_{\tau_{D_{r, R}}},y),$$
where $D_{r, R}= (-R,-r)\cup (r,R)$. Proposition \ref{Greenb}, the dominated convergence theorem, continuity of $\Go$ and quasi-left continuity of $X$ yield the conclusion when we pass $r\to 0$.
\end{proof}

\begin{prop}\label{exittime1}
For $|x|\in (0,R)$ we have
\begin{equation*}
\E^x[\tau_{(-R,R)}\wedge T_0]\le 4R K(x) .
\end{equation*}
\end{prop}
\begin{proof}{By Proposition \ref{Greenb},}
\begin{eqnarray*} \E^x[\tau_{(-R,R)}\wedge T_0]&=&\int^R_{-R} G_{(-R,0)\cup(0,R)}(x,y)dy\leq
\int^R_{-R} \Go(x,y)dy\\ &\leq & 
4RK(x).
\end{eqnarray*}

\end{proof}

\begin{prop} \label{MRestimate1}Let $K$ be non-decreasing on $[0,\infty)$. For $R>0$ and $0<|x|<R$.
$$\frac16\frac{K(x)}{K(R)}\leq {\p^x(\tau_{(-R, R)}<T_0)}.$$
If $\kappa=0$, then additionally
$$ {\p^x(\tau_{(-R, R)}<T_0)}\leq  4\frac{K(x)}{K(R)}.$$

\end{prop}
\begin{proof} Let $0<x< R$. 
By Proposition \ref{Greenb}, Lemma \ref{GreenRegHarm} and subadditivity of $K$,
$$ K(x){\p^{2R}(T_0<\infty)} \leq \Go(x,2R)= \E^x\Go(X_{\tau_{(-R,R)}\wedge T_0},2R)\leq 4 K(R) \p^x(|X_{\tau_{(-R,R)}\wedge T_0}|\ge R).$$
On the other hand, by Lemma \ref{php1}  and subadditivity of $K$,  $${\p^x(\tau_{(-R,R)}<T_0)}\ge \p^{x}(T_0=\infty)= \kappa K(x)= \frac {K(x)}{K(2R)}\p^{2R}(T_0=\infty)\ge \frac12\frac {K(x)}{K(R)}\p^{2R}(T_0=\infty), $$ which combined with the first bound above provide the first estimate.
Moreover, if $\kappa=0$ then
$$ 2K(x) \geq  \Go(x,2R)= \E^x\Go(X_{\tau_{(-R,R)}\wedge T_0},2R)\geq  K(R) \p^x(X_{\tau_{(-R,R)}\wedge T_0}\ge R)
$$
and
$$ 2K(x) \geq  \Go(x,-2R)= \E^x\Go(X_{\tau_{(-R,R)}\wedge T_0},-2R)\geq  K(R) \p^x(X_{\tau_{(-R,R)}\wedge T_0}\le- R)
.
$$
Hence
$$ 4K(x) \geq 
K(R) \p^x(|X_{\tau_{(-R,R)}\wedge T_0}|\ge R).$$
\end{proof}

We also consider the renewal function $V$
   of the (properly normalized) ascending ladder-height process of $X_t$.
The ladder-height process is a subordinator with the Laplace exponent
\begin{equation*}\label{kappa}
 \kappa(\xi)=
\exp\left\{\frac{1}{\pi} \int_0^\infty \frac{ \log {\psi}(\xi\zeta)}{1 + \zeta^2} \, d\zeta\right\}, \quad \xi\ge 0,
\end{equation*}
and $V(x), \,x\ge 0,$ is its potential measure of the interval  $[0,x]$. For $x<0$ we set $V(x)=0$. 
Silverstein studied $V$ and its derivative $V'$ as $g$ and $\psi$ in \cite[(1.8) and Theorem~2]{MR573292}.
The  Laplace transform of $V$ is
\begin{equation}\label{eq:tLV}
\mathcal{L}V(\xi)=\int_0^\infty V(x)e^{-\xi x}dx=\frac{1}{\xi\kappa(\xi)}, \quad \xi>0.
\end{equation}
For instance,
$V(x)= x^{\alpha/2}$ for $x\ge 0$, if $\psi(\xi)= |\xi|^\alpha$ \cite[Example~3.7]{MR2453779}.
The definition of $V$ is rather implicit and properties of $V$ are delicate. In particular
 the decay properties of $V'$ are not yet fully understood.
For a detailed discussion of $V$ we refer the reader to
\cite{BGR3} and \cite{MR573292}.
We have $V(x)=0$ for $x\le 0$ and $V(\infty):=\lim_{r\to \infty}V(r)=\infty$. Also, $V$ is
subadditive:
\begin{equation}\label{subad}
 V(x+y)\le V(x)+V(y), \quad x,\,y \in \R.
\end{equation}
It is known that $V$ is absolutely continuous and harmonic on $(0,\infty)$ for $X_t$. Also
$V^\prime$
is a
positive harmonic function for
$X_t$ on $(0,\infty)$, hence $V$ is actually (strictly) increasing.
For the so-called complete subordinate Brownian motions \cite{MR2978140} $V'$ is monotone, in fact completely monotone, cf. \cite[Proposition 4.5]{MR3098066}. This property was crucial for the development in \cite{MR2923420,MR3131293}, but in general it fails in the present setting cf. \cite[Remark~9]{BGR3}.
One of the important features of the  function $V^\prime$ is the fact that   the Green function of   $(0,\infty)$ can be written as
\begin{equation}\label{Gformula_F}G_{(0,\infty)}(x,y)=\int^{x}_0V'(u)V'( y-x + u)du,\quad 0<x< y.\end{equation}
This follows from \cite[Theorem
VI.20]{MR1406564}.

Let $\psi^*(x)= \sup_ {|u|\le x}  \psi(u), x\ge 0$ be the maximal function of $\psi$.
By \cite[Theorem~2.7]{1998-WHoh-habilitation},
\begin{equation}\label{Psi*ScalingGeneral}
\psi(s
{u})\leq\psi^*(s
u)\le 2
(s^2+1)\psi^*(
u
),
\qquad
s,u
\ge 0.
\end{equation}

Below, in Lemmas \ref {ch1V}-\ref{exit} we collect useful facts which are true for general symmetric L\'evy  processes, which are not compound Poisson.
\begin{lem}[\cite{BGR3}, Proposition 2.4]\label{ch1V}
There is an absolute  constant $C_1\ge 1$ such that
\begin{equation*}
 C^{-1}_1\frac1 {\sqrt{\psi^*(1/r)}} \le V(r)\le  C_1\frac1 {\sqrt{\psi^*(1/r)}}
,\quad r>0.
\end{equation*}
\end{lem}
\begin{lem}[\cite{BGR3}, (2.23) and (2.24)]\label{upper_den} There is an absolute  constant $C_2$  such that

\begin{equation}\label{B0}\p^0(|X_t|\ge r)\le C_2\frac{t}{ V^2(r)},\quad r>0.\end{equation}
and
\begin{equation}\label{B}\nu[r, \infty)\le C_2\frac{1}{ V^2(r)},\quad r>0.\end{equation}
Moreover for any $D\subset B_r$ and $|x|<r/2$,
\begin{equation}\label{B1}\p^x(|X_{\tau_D}|\ge r)\le C_2\frac{\E^x\tau_D}{ V^2(r)},\quad r>0.\end{equation}
\end{lem}

\begin{lem}[\cite{MR3098066}, Theorem 3.1] \label{halfspace} There is an absolute constant $C_3$ such that for $x>0, t>0$,

$$C_3 \(\frac {V(x)}{\sqrt{t}}\wedge1\) \le \p^x(\tau_{(0,\infty)}>t)\le  2 \(\frac {V(x)}{\sqrt{t}}\wedge1\).$$

\end{lem}

\begin{lem} [\cite{MR3007664}, Proposition 3.7]\label{exit} There is an absolute constant $C_4$ such that for any $x\in (0, R),$

 $$C_4 \frac{V(x)}{V(R)}\le  \p^x(X_{\tau_{(0,R)}}\ge R) \le  \frac{V(x)}{V(R)}.$$
 In fact we may take $C_4= \frac{C_3^2}4$.
\end{lem}

\begin{lem} [\cite {MR3165234}, Lemma 1] \label{cos} Let $f:(0, \infty)\mapsto [0,\infty)$ be non-increasing. Then   for $x>0$,

 $$\frac 2{\pi^2} \int^\infty_0[1\wedge(xr)^2]f(r) {dr} \le \int^\infty_0(1-\cos(xr))f(r) {dr}.$$

\end{lem}

For a continuous non-decreasing function $\phi : [0,\infty) \to [0,\infty)$, {such that $\phi(0)=0$ and $\lim_{s\to \infty}\phi(s)=\infty$ and
define
the generalized inverse $\phi^{-1}: {[0,\infty) \to [0, \infty)}$,
\begin{equation*}\label{2s}
\phi^{-1}({u}) = \inf \{{s}
\ge 0 : \phi (s)\ge u\},\quad 0\le u<\infty.
\end{equation*}
 The function $\phi^{-1}$ is non-decreasing and c\`agl\`ad  (left continuous with right-hand side  limits).
Notice that $\phi(\phi^{-1}(u))=u$ for $u\in [0,\infty)$ and $\phi^{-1}(\phi(s))\leq s$ for $s\in [0,\infty)$.
Also, if $\varphi : [0,\infty) \to [0,\infty)$, $\varphi(0)=0$, $c>0$ and $c\phi\le \varphi$, then $\phi^{-1}(u)\ge \varphi^{-1}(cu)$, $u\ge 0$.
Below we often consider the (unbounded) characteristic exponent $\psi$ of a symmetric L\'evy process with infinite L\'evy measure and its  maximal function $\psi^*$,
and denote $$\psi^{-1}=(\psi^*)^{-1}.$$ This short notation is motivated by
the following equality:
\begin{equation*}\label{2sp}
\inf \{{s}
\ge 0 : \psi (s)\ge u\}=\inf \{{s}
\ge 0 : \psi^* (s)\ge u\},\qquad 0\le u<\infty.
\end{equation*}

It is rather natural to assume (relative) power-type behaviour
 for the characteristic exponent $\psi$ of $X$.
To this end we consider
 $\psi$ as a function on $(0,\infty)$.
We say that
$\psi$ satisfies {the} global  weak lower scaling condition
 (WLSC) if there are numbers
$\alpha>0
$ (called  the index of the lower scaling)
and  $\gamma\in(0,1]$,  such that
\begin{equation*}\label{eq:LSC2}
 \psi(\lambda\theta)\ge
\gamma \lambda^{\,\alpha} \psi(\theta)\quad \mbox{for}\quad \lambda\ge 1, \quad
\theta>0.
\end{equation*}
In short we write $\psi\in\WLSC{\alpha}{\gamma}$ or $\psi\in {\rm WLSC}$.
The global weak upper scaling condition
 (WUSC) means that
there are numbers $\beta <2$ (called  the index of the upper scaling)
and $\rho{\in [1,\infty)}$ such that
\begin{equation*}\label{eq:USC2}
 \psi(\lambda\theta)\le
\rho\lambda^{\,\beta} \psi(\theta)\quad \mbox{for}\quad \lambda\ge 1, \quad\theta>0.
\end{equation*}
In short, $\psi\in\WUSC{\beta}{\rho}$ or $\psi\in{\rm WUSC}$.
Similarly,

 We call $\alpha$, $\gamma$, $\beta$, $\rho$ the scaling characteristics of $\psi$ or simply the scalings. In most of our  results we assume only the lower scaling condition.

Here are further remarks from \cite{MR3165234}:
We have $\psi\in\WLSC{\alpha}{\gamma}$ if and only if $\psi(\theta)/\theta^\alpha$
is  comparable to a non-decreasing function on $(0,\infty)$, and $\psi\in$WUSC($\gamma$,$\rho$) if and only if $\psi(\theta)/\theta^\beta$
is  comparable to a non-increasing function on $(0,\infty)$, see \cite[Lemmas~8, 9 and 11]{MR3165234}.

We are thus led to the behavior of {$\psi^{-1}$}.
{\begin{lem}If $\psi\in\WLSC{\alpha}{\gamma}$, then
\begin{equation}\label{psiInvScal}\psi^{-1}\in\WLSC{1/2}{(\gamma/48^2)^{1/\alpha}}\cap\WUSC{1/\alpha}{(48^3/\gamma^2)^{1/\alpha}}.\end{equation}  
\end{lem}}
\begin{proof}
{The proof is similar to the proof of \cite[Lemma 18]{MR3165234}, where  unimodal L\'{e}vy processes were considered. Let $W(x)=\int_{\R}(1\wedge(x s)^2)\nu(ds)+ \sigma^2 x^2$. Since $\psi$ is unbounded the function $W$ is increasing on $(0,\infty)$.  Moreover, by
\cite[Lemma 4]{MR3225805}
$$\frac{1}{2}\psi^*(x)\leq W(x)\leq 24 \psi^*(x),\quad x\geq0.$$
Since  $\psi\in\WLSC{\alpha}{\gamma}$ then  $\psi^*\in\WLSC{\alpha}{\gamma}$ and $W\in \WLSC{\alpha}{\frac{\gamma}{48}}$. Now, we can repeat the arguments of \cite[Lemma 18]{MR3165234} to arrive at \eqref{psiInvScal}.
}
\end{proof}
\begin{lem}\label{KWLSC} If $\psi\ge a \psi^*$, then for $x>0$,
$$\frac{2}{\pi^3}\int^\infty_{1/x}\frac{ds}{\psi^*(s)} \le K(x)\le \frac{10}{\pi a} \int^\infty_{1/x}\frac{ds}{\psi^*(s)}.$$

Moreover,  $\psi\in\WLSC{\alpha}{\gamma}$, $\alpha>1$ if and only if  $$K(x)\approx \frac{1}{|x|\psi(1/x)}\approx \frac{V^2(|x|)}{|x|},\ x\neq0 \quad \text{and}\quad \psi^*\approx \psi.$$
 The comparability constants depend only on the scalings. In this case $K\in\WLSC{\alpha-1}{\gamma_1}$, where $\gamma_1>0$ depends on the scalings of $\psi$.
\end{lem}
\begin{proof} Let  $\tilde{K}(x)=\frac1\pi\int_0^\infty(1-\cos xs)\frac{1}{\psi^*(s)}ds.$
Observe that $\tilde{K}$  is the L\'{e}vy-Khintchine exponent of some isotropic unimodal L\'{e}vy process with the L\'evy density $\frac1{2\pi\psi^*(|s|)}$. By Lemma \ref{cos} and the inequality $1-\cos s\le s^2/2$ we have
 $$\frac 2{\pi^2} \int^\infty_0(1\wedge(xs)^2)\frac{ds}{\psi^*(s)} \le \pi\tilde{K}(x)\le \frac {x^2}2 \int^{1/x}_0s^2\frac{ds}{\psi^*(s)}+  2\int^\infty_{1/x}\frac{ds}{\psi^*(s)}.$$
 Moreover, by \eqref{Psi*ScalingGeneral},  for $0<xs\leq 1$, 
$$(xs/2)^2\psi^*(2/x)\le 4  \psi^*( s),$$
which implies that
$$\frac {x^2}2 \int^{1/x}_0s^2\frac{ds}{\psi^*(s)}\le \frac8{x\psi^*(2/x)}\le
 8\,\int^{2/x}_{1/x}\frac{ds}{\psi^*(s)} \le 8\,\int^{\infty}_{1/x}\frac{ds}{\psi^*(s)}.$$
Hence,

\begin{equation}\label{Ktil1}\frac 2{\pi^3} \int^\infty_{1/x}\frac{ds}{\psi^*(s)} \le \tilde{K}(x) \le \frac{10}\pi \int^\infty_{1/x}\frac{ds}{\psi^*(s)}.\end{equation}
Since  $ \tilde{K}(x) \le   K(x)\le \frac 1a\,\tilde{K}(x)$ we get the first conclusion.


Suppose that $\psi\in\WLSC{\alpha}{\gamma}$, $\alpha>1$. Then $\psi\ge \gamma \psi^*$ and $\psi^*\in\WLSC{\alpha}{\gamma}$ as well.
  Hence, {for $x>0$},
\begin{equation}\label{Ktil2}\frac 1{16}\frac1{x\psi^*(1/x)}\le \frac1{x\psi^*(2/x)} \le \int_{1/x}^\infty\frac{dr}{\psi^*(r)}\le \frac1{\gamma(\alpha-1)x\psi^*(1/x)},\end{equation}
which shows that $ K(x)\approx  \frac1{x\psi(1/x)}$ with the comparability constant dependent on the scaling characteristics. Also, it is evident that $K$ satisfies  the weak lower scaling condition  with index $\alpha-1$.
Furthermore, by Lemma \ref{ch1V} we get $K(x)\approx \frac{V^2(|x|)}{|x|}$.

Next we assume that $\frac{1}{\psi^*(|x|)}\approx \frac{K(1/x)}{|x|}$.  By \cite[Theorem 26]{MR3165234} it is equivalent to the fact that
 $K(x)$ satisfies the global weak lower and upper scaling conditions with indices $0<\delta\leq \beta< 2$, respectively.   This  implies that
 $xK(x) $ satisfies the global weak lower scaling condition  with index $\alpha=\delta+1$.
Equivalently, $\psi^*$ satisfies  the weak lower scaling condition  with index $\alpha=\delta+1>1$. The proof is completed.

\end{proof}

The following technical lemma is the main tool in estimating the tail function of $T_0$ via its Laplace transform. Recall that $\tilde{K}(x)=\frac1\pi\int_0^\infty(1-\cos xs)\frac{1}{\psi^*(s)}ds.$

\begin{lem} \label{potEst}For any $\lambda>0$,
%
\begin{eqnarray}\label{ulam12}
u^\lambda(0)&\ge&\frac{1}{4}\tilde{K}\(\frac{1}{\psi^{-1}(\lambda)}\){\geq\frac{1}{32{\pi^3}}\frac{\psi^{-1}(\lambda)}{\lambda}}.
\end{eqnarray}
If $a\psi^*(x)\leq \psi(x), x\ge 0 $, then
\begin{eqnarray*}
\frac{a}{4}K\(\frac{1}{\psi^{-1}(\lambda)}\)\le u^\lambda(0)\le \frac{3\pi^2}{2a} K\(\frac{1}{\psi^{-1}(\lambda)}\).
\end{eqnarray*}
For $x\psi^{-1}(\lambda)\le 1$, $ x\ge 0 $,
$$K^\lambda(x)\ge \frac{a}{10\pi^2} K(x). $$
\end{lem}
\begin{proof}
 By \eqref{Psi*ScalingGeneral}, $\lambda=\psi^*\(\psi^{-1}(\lambda)ss^{-1}\)\le (2s^{-2}+2)\psi^*\(\psi^{-1}(\lambda)s\)$, hence
\begin{eqnarray*}
u^\lambda(0)&=&\frac{1}{\pi}\int^\infty_0\frac{dr}{\lambda+\psi(r)}=\frac{\psi^{-1}(\lambda)}{\pi}\int^\infty_0\frac{ds}{\lambda+\psi\(\psi^{-1}(\lambda)s\)}\\&\geq&\frac{\psi^{-1}(\lambda)}{\pi}\int^\infty_0\frac{ds}{\lambda+\psi^*\(\psi^{-1}(\lambda)s\)}\geq \frac{\psi^{-1}(\lambda)}{\pi}\[\int^1_0\frac{ds}{2\lambda}+\int^\infty_1\frac{ds}{2\psi^*\(\psi^{-1}(\lambda)s\)}\]\\
&\geq&\frac{\psi^{-1}(\lambda)}{2\pi}\[\int^1_0\frac{s^2ds}{4\psi^*(\psi^{-1}(\lambda)s)}+\int^\infty_1(1-\cos s) \frac{ds}{2\psi^*\(\psi^{-1}(\lambda)s\)}\]\\
&\geq&\frac{\psi^{-1}(\lambda)}{4\pi}\int^\infty_0(1-\cos s) \frac{ds}{{\psi^*}\(\psi^{-1}(\lambda)s\)}=\frac{1}{4}{\tilde{K}}\(\frac{1}{\psi^{-1}(\lambda)}\).
 \end{eqnarray*}
By \eqref{Ktil1} and \eqref{Ktil2} we have $\tilde{K}(x)\geq \frac{1}{8\pi^3}\frac1{x\psi^*(1/x)}$, which  implies the second inequality in \eqref{ulam12}.

Now assume that  $a\psi^*(x)\leq \psi(x), \ x\ge 0 $. {Then we have
$$u^\lambda(0)\geq \frac{1}{4}\tilde{K}\(\frac{1}{\psi^{-1}(\lambda)}\)\geq \frac{a}{4}K\(\frac{1}{\psi^{-1}(\lambda)}\).$$
}
To obtain the upper bound we apply Lemma \ref{cos} with $f(r)= \frac{1}{\psi^*\(\psi^{-1}(\lambda)r\)}$ to get

\begin{eqnarray*}
u^\lambda(0)&\leq& \frac{\psi^{-1}(\lambda)}{\pi}\[\int^1_0\frac{ds}{\lambda}+\int^\infty_1\frac{ds}{\psi\(\psi^{-1}(\lambda)s\)}\]\\
&=&\frac{\psi^{-1}(\lambda)}{\pi}\[3\int^1_0\frac{s^2ds}{\psi(\psi^{-1}(\lambda))}+\int^\infty_1 \frac{ds}{\psi\(\psi^{-1}(\lambda)s\)}\]\\
&\leq&\frac{3\psi^{-1}(\lambda)}{a\pi}\int^\infty_0(s^2\wedge1) \frac{ds}{\psi^*\(\psi^{-1}(\lambda)s\)}\\
&\leq& \frac{3\pi^2}{2a}\frac{\psi^{-1}(\lambda)}{\pi}\int^\infty_0 (1-\cos s) \frac{ds}{\psi^*\(\psi^{-1}(\lambda)s\)}\\
&\leq& \frac{3\pi^2}{2a} K\(\frac{1}{\psi^{-1}(\lambda)}\).
 \end{eqnarray*}

For $0<x\psi^{-1}(\lambda)\leq 1$, applying Lemma \ref{cos} with $f(s)= \frac{1}{\lambda+\psi^*\(
s\)}$, we obtain
\begin{eqnarray*}
\pi K^\lambda(x)&=&\int^\infty_0(1-\cos(xs))\frac{ds}{\lambda+\psi(s)}\geq\int^\infty_0(1-\cos(x s))\frac{ds}{\lambda+\psi^*\( s \)}\\
&\geq&\frac2{\pi^2}\int^\infty_0(1\wedge(x s)^2)\frac{ds}{\lambda+\psi^*\( s \)}\geq\frac1{\pi^2}\int^\infty_{1/x}\frac{ds}{\psi^*\( s \)}\\
&\geq&\frac{a}{10\pi} K(x),
\end{eqnarray*}
where the last step follows from Lemma \ref{KWLSC}.
\end{proof}
For two functions $g, f$ we write  $g(x)\cong  f(x), x\to x_0,$ if $\lim_{x\to x_0} g(x)/  f(x)=1.$

\begin{lem} \label{potential}Suppose that  $\psi(r)$ is regularly varying at $0$ with index $1<\delta\le 2$. Then

$$  u^\lambda(0)  \cong   \frac {\psi^{-1}(\lambda)}{\lambda} \frac{1}{\delta\sin\frac{\pi}{\delta}},\quad  \lambda\to 0^+$$
and $u^\lambda(0)$ is regularly varying at $0$ with index $1/\delta-1$.

 If $\psi(r)$ is regularly varying  at $0$ with index $1$, then

$$u^\lambda(0)
\cong \frac{1}{\pi}\int^\infty_{\psi^{-1}(\lambda)}\frac{ds}{\psi(s)}, \quad \lambda\to 0^+$$
 and $u^\lambda(0)$ is slowly varying.
\end{lem}
\begin{proof}  Assume that $\psi(s)$ is regularly varying with index $1<\delta\le 2$.
\begin{eqnarray*}
u^\lambda(0)&=&\frac{1}{\pi}\int^\infty_0\frac{ds}{\lambda+\psi(s)}=\frac{1}{\pi}\int^1_0 \frac{ds}{\lambda+\psi(s)} +\frac{1}{\pi}\int^\infty_1 \frac{ds}{\lambda+\psi(s)}\\&=&\frac{\psi^{-1}(\lambda)}{\lambda\pi}\int^{\frac1{\psi^{-1}(\lambda)}}_0\frac{dw}{1+\frac{\psi\(\psi^{-1}(\lambda)w\)}{\psi^*\(\psi^{-1}(\lambda)\)}}+\frac{1}{\pi}\int^\infty_1 \frac{ds}{\lambda+\psi(s)}.
\end{eqnarray*}

The second integral converges to $\frac{1}{\pi}\int^\infty_1 \frac{dr}{\psi(r)}<\infty$ and since $\frac{\lambda}{\psi^{-1}(\lambda)}\to 0$ its has no contribution  to the limit. Note that $\psi^*(u)\cong \psi(u), u \to 0$ (see \cite[Theorem 1.5.3]{MR1015093}). Since $\psi(x)>0, x\ne 0$, by Potter's lemma \cite[Theorem 1.5.6]{MR1015093}, and  continuity  of $\psi$ and $\psi^*$, for $1<\delta^*<\delta$, we can find a constant $c=c( \delta^*, \psi)>0$  such that  for $\lambda<1,  \psi^{-1}(\lambda)s<1, s>1$,  $$\frac{\psi\(\psi^{-1}(\lambda)s\)}{\psi^*\(\psi^{-1}(\lambda)\)}\ge c s^{\delta^*}.$$  By  the dominated convergence theorem
$$\lim_{\lambda\to 0 }\frac{\lambda}{\psi^{-1}(\lambda)}u^\lambda(0)=\frac{1}{\pi}\int^\infty_0\frac{ds}{1+s^\delta}=\frac{\Gamma(1/\delta)\Gamma(1-1/\delta)}{\pi\delta}=\frac{1}{\delta\sin\frac{\pi}{\delta}}.$$

Next, let  $\psi(r)$ be regularly varying at $0$ with index $1$.  Let $L(u)= \int^\infty_{u}\frac{dr}{\psi(r)}$. This function is slowly varying at $0$ \cite[Proposition 1.5.9a]{MR1015093}.
Note that $\int^{\psi^{-1}(\lambda)}_0\frac{dr}{\lambda+\psi(r)}\le \frac{\psi^{-1}(\lambda)}\lambda= \frac{\psi^{-1}(\lambda)}{\psi^*(\psi^{-1}(\lambda))}$. Due to regular variation of $\psi^2$ with index $2$ we have
$$\int^\infty_{\psi^{-1}(\lambda)}\frac{dr}{\psi(r)}-\int^\infty_{\psi^{-1}(\lambda)}\frac{dr}{\lambda+\psi(r)}\le \lambda  \int^\infty_{\psi^{-1}(\lambda)}\frac{dr}{\psi^2(r)}
\cong \frac {\psi^{-1}(\lambda)}{\lambda},\quad \lambda \to 0^+.  $$
Hence it is enough to prove that $$\frac{ u/\psi^*(u)}{L(u)}\to 0,\quad u\to 0^+.$$ Let $a>1$. Then
$$\frac{L(u)-L(au)}{{L}(u)}\ge \frac {(a-1)u} {L(u)\psi^*(au)}.$$
Since $L$ varies slowly  the left hand side converges to $0$ so the proof is completed.
\end{proof}

\section{Hitting times of points}

In this section we examine the tail function $ \p^x(T_0>t)$ under various assumptions on $\psi$ or $K$. Under the monotonicity of $K$ we find the lower and upper bounds of the tail function. On the other hand comparability of $\psi$ and $\psi^*$ is another source  of the estimates via  approximate inversion of the Laplace transform. We also derive asymptotics of $ \p^x(T_0>t)$ if $t\to \infty$ by applying Tauberian theorems.

\begin{prop}\label{0hit}
We have for any $t>0$ and $x\in \R$
$$ \p^x(T_0>t)\leq {\[7\frac{K(x)}{\tilde{K}\(\frac{1}{\psi^{-1}(1/t)}\)}\]\wedge 1\leq  \[51\pi^3\frac{K(x)}{t\psi^{-1}(1/t)}\]\wedge 1}.$$
If additionally $\psi(x)\geq a\psi^*(|x|)$ for $x\in \R$, then
 $$ \p^x(T_0>t)\leq  \[\frac{7}{a}\frac{K(x)}{K\(\frac{1}{\psi^{-1}(1/t)}\)}\]\wedge 1.$$
\end{prop}
\begin{proof}
Observe that for the Laplace transform $\p^x(T_0>\cdot)$ we have
$$\mathcal{L}(\p^x(T_0>\cdot))(\lambda)=\frac{1}{\lambda}\[1-\E^xe^{-\lambda T_0} \]=\frac{1}{\lambda}\frac{u^\lambda(0)-u^\lambda(x)}{u^\lambda(0)}\le \frac{1}{\lambda}\frac{K(x)}{u^\lambda(0)}
.$$
It follows from {\eqref{ulam12}} that
$$\mathcal{L}(\p^x(T_0>\cdot))(\lambda)\le {\frac{4}{\lambda}\frac{K(x)}{\tilde{K}(\frac{1}{\psi^{-1}(\lambda)})}\leq 32\pi^3\frac{K(x)}{\psi^{-1}(\lambda)} }$$ in general case, while under the assumption $\psi(x)\geq a\psi^*(|x|)$,

$$\mathcal{L}(\p^x(T_0>\cdot))(\lambda)\le \frac{4 }{a\lambda}\frac{K(x)}{K\(\frac{1}{\psi^{-1}(\lambda)}\)}.$$
By \cite[Lemma 5]{MR3165234} we have
$$ \p^x(T_0>t)\leq  {\frac{4e}{e-1}\frac{K(x)}{\tilde{K}\(\frac{1}{\psi^{-1}(1/t)}\)} \leq 32\pi^3}\frac{e}{e-1}\frac{K(x)}{t\psi^{-1}(1/t)}$$ in general case, and in the other considered  case
$$ \p^x(T_0>t)\leq  \frac{e}{e-1}\frac{4}{a}\frac{K(x)}{K\(\frac{1}{\psi^{-1}(1/t)}\)}, $$ which ends the proof.
\end{proof}

\begin{prop}\label{optimaltail}Let $K$ be non-decreasing on $[0,\infty)$ and $\kappa=0$. Then 
$$ \p^x(T_0>t)\leq   \[8\frac{K(x)}{K(R_t)}\]\wedge 1,$$
where $R_tK(R_t)=t$.
\end{prop}
\begin{proof}Let $R, t>0$.
We have
$$\p^x(T_0>t)\leq \p^x(\tau_{(-R,R)}\wedge T_0>t)+\p^x(\tau_{(-R,R)}<T_0).$$
Let $|x|<R$. By Chebyshev's inequality and Proposition
\ref{exittime1} we obtain $$\p^x(\tau_{(-R,R)}\wedge T_0>t)\leq
\frac{\E^x\tau_{(-R,R)}\wedge T_0}{t}\leq4 \frac{RK(x)}{t},$$ while by  Lemma
\ref{MRestimate1},
$$\p^x(\tau_{(-R,R)}<T_0)\leq 4 \frac{K(x)}{K(R)}.$$
Setting  $RK(R)=t$ we obtain the conclusion.


\end{proof}

\begin{lem}\label{0hitLB}Let $K$ be non-decreasing on $[0,\infty)$. For $x\in \R, t>0$,
$$\p^x(T_0>t)\geq \frac {C_3}{6C_1}\( \frac{K(x)}{ K(1/\psi^{-1}(1/t))}\wedge 1\).$$
\end{lem}
\begin{proof}

For $x\geq 1/\psi^{-1}(1/t)$ we have, by Lemma \ref{ch1V},  $V(x)\ge V({1/}\psi^{-1}(1/t))\ge \frac {\sqrt{t}}{C_1}$. Hence   by Lemma \ref{halfspace},
\begin{equation}\label{estimate0}\p^x(T_0>t)\ge \p^x(\tau_{(0,\infty)}>t)\ge  \frac {C_3}{C_1}.\end{equation}
Let $R=1/\psi^{-1}(1/t)$. Then for $0< x<R$, by Proposition \ref{MRestimate1} and the strong Markov property
\begin{eqnarray*}\p^x(T_0>t)&\geq& \E^x \{|X_{\tau_{(-R,R)}\wedge T_0}|\ge R ; \p^{X_{\tau_{(-R,R)}\wedge T_0}}(T_0>t)\}\\
&\geq& \frac {C_3}{C_1} \p^x \{|X_{\tau_{(-R,R)}\wedge T_0}|\geq R\}\ge \frac {C_3}{6C_1} \frac{K(x)}{K(R)}.\end{eqnarray*}
The proof is completed.

\end{proof}

The assumption about monotonicity of $K$ can be removed if we assume the lower scaling condition of $\psi$.
\begin{prop}\label{LBWLSC}

Let  $\psi\in\WLSC{\alpha}{\gamma}$, $\alpha>1$.  For $x\in \R, t>0$, $$\p^x(T_0>t) \geq  c\(\frac{K(x)}{K(1/\psi^{-1}(1/t))}\wedge 1\),$$
where $ c$ depends only on the scalings.
\end{prop}
\begin{proof}Since $\psi\in\WLSC{\alpha}{\gamma}$ we have $\gamma\psi^*(r)\leq \psi(r)$ for any $r\geq 0$.
By Lemma \ref{potEst},
$$\frac{2\gamma}{ 3\pi^2 }\frac{K^\lambda(x)}{ K\(\frac{1}{\psi^{-1}(\lambda)}\)}\leq \lambda\mathcal{L}(\p^x(T_0>\cdot))(\lambda)\leq \frac{4}{\gamma}\frac{K^\lambda(x)}{  K\(\frac{1}{\psi^{-1}(\lambda)}\)}.$$
This, Lemma \ref{KWLSC} {and \eqref{psiInvScal}} imply for $\lambda>0$ and $s>1$,
\begin{eqnarray} \label{WUSC0}
\frac{\mathcal{L}(\p^x(T_0>\cdot))(\lambda s)}{\mathcal{L}(\p^x(T_0>\cdot))(\lambda)}&\leq& \frac{6\pi^2}{\gamma^2}\frac{1}{s}\frac{K^{\lambda s}(x)}{K^\lambda(x)}\frac{K\(\frac{1}{\psi^{-1}(\lambda)}\)}{K\(\frac{1}{\psi^{-1}(\lambda s)}\)}\le c_1 \frac{\psi^{-1}(\lambda)}{\psi^{-1}(\lambda s)}\leq c_2 s^{-1/2},
\end{eqnarray}
where $c_2$ depends only on the scalings.
 Hence, by \cite[Lemma 13]{MR3165234} there exists a constant $c_3$ that depends only on the scalings such that
$$\p^x(T_0>t) \geq c_3 \frac{K^{1/t}(x)}{ K\(\frac{1}{\psi^{-1}(1/t)}\)}.$$
 For $t\ge 1/\psi^*(1/x) $, by Lemma  \ref{potEst},    $K^{1/t}(x)\ge \frac{\gamma}{10\pi^2} K(x), $ which gives the conclusion in this case.
  We complete the proof  by applying \eqref{estimate0} for $t< 1/\psi^*(1/x) $ .


\end{proof}


From the above lower and upper bounds we derive two corollaries providing two sided sharp estimates.
\begin{cor}\label{unimodal}
 If $X$ is unimodal then for $x\in\R$ and $t>0$,
$$ \p^x(T_0>t)\approx \frac{K(x)}{K(1/\psi^{-1}(1/t))}\wedge 1.$$
The comparability constant is absolute.
\end{cor}

\begin{proof}
If $X$ is unimodal then $K$ is increasing and $\psi \ge \pi^{-2}\psi^*$ (see \cite[Proposition 2]{MR3165234}), hence the upper bound follows from Proposition \ref{0hit}, while the lower bound is a consequence of Lemma \ref{0hitLB}.

\end{proof}



\begin{cor}\label{hitpropWLSC}
Let $\psi\in\WLSC{\alpha}{\gamma}$, $\alpha>1$. Then  for $x\in\R$ and $t>0$,
$$ \p^x(T_0>t)\approx \frac{K(x)}{K(1/\psi^{-1}(1/t))}\wedge 1 \approx \frac{1}{t\psi^{-1}(1/t)|x|\psi(1/x)}\wedge 1.$$
The comparability constants depend on the scalings.\end{cor}

\begin{proof}  By Lemma \ref{KWLSC} we have $K(x)\approx \frac{1}{|x|\psi(1/x)}$ and the conclusion follows immediately from Proposition \ref{0hit} and Proposition \ref{LBWLSC}.   \end{proof}

\begin{exmp}  \label{exp1} Let $\psi(x)= |x|+x^2$. The corresponding process, which is the sum of the Cauchy process and independent Brownian motion, is obviously unimodal. Then
$$K(x)\approx \int^\infty_{1/x}\frac{dr}{\psi(r)}\approx \log (1+x)$$
{and by} Corollary \ref{unimodal},
$$ \p^x(T_0>t)\approx  \frac{ \log (1+x)}{\log (1+\sqrt{t})}\wedge 1.$$
  \end{exmp}

\begin{exmp}\label{exp2}{
We consider $\psi(x)=2\int^{1}_0(1-\cos(|x|s))\nu(ds)$, where $\nu$ is singular.
Namely, let  $\nu(ds)=\sum^\infty_{k=1}\delta_{1/k}(ds)\(k^{\alpha}-(k-1)^{\alpha}\)$
or $\nu(ds)=s^{-\beta}F(ds)$, where $\alpha\in(1,2)$ and $\beta=\alpha+\log 2/\log 3$ and $F$ is the standard Cantor measure on [0,1].
In both cases we claim that
$\psi(x)\approx |x|^{\alpha}\wedge x^2$. Indeed, by the integration by parts and \cite[Lemma 2]{GK2004} to verify the claim in the second case, we obtain that $\int^{|x|^{-1}}_0s^2\nu(ds)\approx 1\wedge |x|^{\alpha-2}$. Moreover $\int_{|x|^{-1}}^\infty \nu(ds)\leq c |x|^{\alpha}$ for $|x|\geq 1$ and $\int_{|x|^{-1}}^\infty \nu(ds)=0$ for $|x|<1$.
Since
$$ 2(1-\cos 1)|x|^2\int^{|x|^{-1}}_0s^2\nu(ds) \leq \psi(x)\leq |x|^2\int^{|x|^{-1}}_0s^2\nu(ds)+2\int_{|x|^{-1}}^\infty \nu(ds)$$
we get the claim.
Hence $\psi\in \WLSC{\alpha}{\gamma}$ for some $0<\gamma\le 1$. Applying  Corollary \ref{hitpropWLSC} we obtain
$$ \p^x(T_0>t)\approx \frac{|x|^{\alpha-1}\vee |x|}{t^{1-1/\alpha}\vee t^{1/2}}\wedge 1.$$}
Since the L\'evy measure is singular   the process with symbol $\psi(x)$ can not be unimodal. This illustrates that Corollary \ref{hitpropWLSC} does not follow from Corollary \ref{unimodal}.
\end{exmp}

\begin{rem} There are recent results obtained by Juszczyszyn and Kwa\'{s}nicki \cite{2014arXiv1403.3714J}  where not only the behaviour of the tail function $\p^x(T_0>t)$ was described but also its derivatives. Their assumptions on the process were much more restrictive than ours. They assumed complete monotonicity of the L\'evy density and some additional property of the first two  derivatives of the symbol  of  the process. The  results of our paper regarding the tail functions are more general, however our methods do now allow us to treat the derivatives. The processes from Example \ref{exp1} and \ref{exp2}  do not satisfy the assumptions of \cite{2014arXiv1403.3714J}. In Example \ref{exp1} the symbol  $\psi$ fails  the   requirements of  \cite{2014arXiv1403.3714J}, while the L\'evy measures in  Example \ref{exp2}   are singular.
\end{rem}

Now we turn to asymptotics of the tail function  when $t\to \infty$   not only  in the case of hitting $\{0\}$ but for hitting arbitrary compact set as well.

\begin{prop}\label{asymp}
Let  $\psi$ be regularly varying at $0$ with index $\delta\in(1,2]$. Then for a compact set $B$ such that $0\in B$ we have for $x\in\R$,
$$\lim_{t\to\infty}\[t\psi^{-1}(1/t)\p^x(T_B>t)\]=\frac{\delta \Gamma\(1-\frac 1 \delta\)\sin^2\frac{\pi}{\delta}}{\pi}\[K(x)-\E^xK(X_{T_B})\].$$

If  $\psi$ is regularly varying at $0$ with index $1$ then there is a function $L(u)$ slowly varying at $0$ such that

$$\lim_{t\to\infty}\[L(1/t)\p^x(T_B>t)\]=K(x)-\E^xK(X_{T_B}).$$
We can take   $L(u)= \frac{1}{\pi}\int^\infty_{\psi^{-1}(u)}\frac{dr}{\psi(r)},\  u>0$.
\end{prop}
\begin{proof}
Observe that
$$\mathcal{L}(\p^x(T_B>\cdot))(\lambda)=\frac{1}{\lambda}\[1-\E^xe^{-\lambda T_B} \]
.$$
Since $0\in B$,  $0$ is regular for $B$. By symmetry $G^\lambda_{B^c}(x,0)=0$. Hence
\begin{eqnarray}
\lambda u^\lambda(0)\mathcal{L}(\p^x(T_B>\cdot))(\lambda)&=&u^\lambda(0)-u^\lambda(x) -\E^xe^{-\lambda T_B}\[u^\lambda(0)-u^\lambda(X_{T_B})\]\nonumber\\&&+u^\lambda(x)-\E^xe^{-\lambda T_B}(u^\lambda(X_{T_B}))\nonumber\\
&=&u^\lambda(0)-u^\lambda(x) -\E^xe^{-\lambda T_B}\[u^\lambda(0)-u^\lambda(X_{T_B})\]+G^\lambda_{B^c}(x,0)\nonumber\\
&=&K^\lambda(x) -\E^xe^{-\lambda T_B}K^\lambda(X_{T_B})\label{laplaceComp}.
\end{eqnarray}
Since $K$ is continuous and $B$ is compact, by the dominated convergence theorem and Lemma \ref{php1},
$$\lim_{\lambda\to0}\lambda u^\lambda(0)\mathcal{L}(\p^x(T_B>\cdot))(\lambda)=K(x)-\E^xK(X_{T_B}).$$

Let  $\psi$ be regularly varying at $0$ with index $\delta\in(1,2]$. By Lemma \ref{potential},
$$\lim_{\lambda\to 0} \psi^{-1}(\lambda)\mathcal{L}(\p^x(T_B>\cdot))(\lambda)=\delta \sin\(\frac{\pi}{\delta}\)[ K(x) -\E^x K(X_{T_B})].$$
Define $U(s)=\int^s_0\p^x(T_{B}>t)dt$.  We have $$\mathcal{L}U(\lambda)=\frac{1}{\lambda}\mathcal{L}(\p^x(T_B>\cdot))(\lambda).$$
Hence
$$\lim_{\lambda\to 0} \lambda\psi^{-1}(\lambda)\mathcal{L}U(\lambda)=\delta \sin\(\frac{\pi}{\delta}\) [K(x)-\E^x K(X_{T_B})].$$
Since $\psi^{-1}$ is regularly varying at $0$ with index $1/\delta$ the Tauberian theorem (\cite[Theorem 1.7.1]{MR1015093}) implies
$$\lim_{t\to \infty} \psi^{-1}(1/t)U(t)=\frac{\delta \sin\(\frac{\pi}{\delta}\)}{\Gamma(1+1/\delta)} [K(x)-\E^x K(X_{T_B})].$$
By the monotone density theorem (\cite[Theorem 1.7.2]{MR1015093}),
 $$\lim_{t\to\infty}\[t\psi^{-1}(1/t)\p^x(T_B>t)\]=\frac{\delta \sin\frac{\pi}{\delta}}{\Gamma(1/\delta)}[K(x)-\E^x K(X_{T_B})].$$

 If $\psi$ is regularly varying  at $0$ with index $1$,  then by Lemma \ref{potential},    $$u^\lambda(0)\cong \frac{1}{\pi}\int^\infty_{\psi^{-1}(\lambda)}\frac{dr}{\psi(r)}= L(\lambda),$$ where
 $L(\lambda)$ is slowly varying at $0$. Hence
 $$\lim_{\lambda\to0}\lambda L(\lambda)\mathcal{L}(\p^x(T_B>\cdot))(\lambda)=K(x)-\E^xK(X_{T_B}).$$
By the monotone density theorem 
 $$\lim_{t\to\infty}L(1/t)\p^x(T_B>t)=K(x)-\E^x K(X_{T_B}).$$
\end{proof}

\begin{cor}\label{0hit_a}Let   $\psi$ be regularly varying at $0$ with index $\delta\in[1,2]$. Then for $x\in\R$,
$$\lim_{t\to\infty}\[t\psi^{-1}(1/t)\p^x(T_0>t)\]=\frac{\delta \Gamma\(1-\frac 1 \delta\)\sin^2\frac{\pi}{\delta}}{\pi}K(x), \quad \delta>1,$$
and

 $$\lim_{t\to\infty}L(1/t)\p^x(T_0>t)=K(x), \quad \delta=1,$$
where $L(u)= \frac{1}{\pi}\int^\infty_{\psi^{-1}(u)}\frac{dr}{\psi(r)}$.

\end{cor}
\begin{proof} Since $ \E^x K(X_{T_0})=0$ it follows from Proposition \ref{asymp}.

\begin{rem} We again compare   \cite{2014arXiv1403.3714J} with our results with regard to asymptotics of the tail function of $T_0$.    For example the case
 $\psi(x)= |x|+|x|^2$ is not covered in \cite{2014arXiv1403.3714J}. In Example \ref{exp2} we provided sharp estimates of $\p^x(T_0>t)$ and Corollary \ref{0hit_a} exhibits the asymptotics at infinity.
 Note that $\psi^{-1}(u)= \frac u{\sqrt{u+\frac14}+ \frac12}\cong u, \quad u\to 0$.  Next,
 $$ L(u)=\frac{1}{\pi}\int^\infty_{\psi^{-1}(u)}\frac{dr}{r+r^2}=\frac1\pi \left[\log \left( u +{\sqrt{u+\frac14}+ \frac12}\right)-\log  u\right]\cong -\frac{\log  u}\pi, \quad u\to 0^+.$$
Then, by the second part of  Corollary \ref{0hit_a},
$$\lim_{t\to\infty}\log t \, \p^x(T_0>t)=\pi K(x),$$
where
$K(x) \approx \log (1+x).$


 \end{rem}
\end{proof}

The next example illustrates that the decay of the tail function of $T_0$ can be very slow. Note that the intensity of  small  jumps of the process below is larger than
the corresponding intensity of the Cauchy process while it is smaller than the corresponding intensity for any symmetric $\alpha$-stable  process, $\alpha>1$.
Therefore the considered process   is in some sense between the Cauchy process and any symmetric $\alpha$-stable  processes, $\alpha>1$. Note that the Cauchy process hits  points with $0$ probability, while   $\alpha$-stable  processes, $\alpha>1$,  hit points with  probability $1$.

\begin{exmp}\label{LCauchy}Let $\nu(r)= \frac {\log^2(2+1/|r|) \log(2+|r|)} {r^2 }, \ r\in\R$. By Lemma \ref{cos}, $$\psi(x)\approx \int^\infty_0(1\wedge(xr)^2)\nu(r){dr}= x^2 \int^{1/x}_0r^2\nu(r){dr}+\int_{1/x}^\infty\nu(r){dr}.$$ Elementary
calculations  show that
$$\psi(x)\approx x  {\log^2(2+x) \log(2+1/x)},\quad x\ge0$$ and
$$\psi^{-1}(x)\approx  \frac{x}{\log^2(2+x) \log(2+1/x)},\quad x\ge0.$$
It is clear that $\psi$ can not have the weak lower scaling property with any $1<\alpha\leq2$, but $\psi= \psi^*$ since $\nu(r)/r$ is decreasing on $(0,\infty)$. Hence, by Lemma \ref{KWLSC},
$$K(x)\approx \int^\infty_{1/x}\frac{dr}{\psi(r)}\approx \frac{\log\log (2+x)}{\log (2+1/x)}, \quad x>0.$$
Moreover,
$\psi(x)\cong  c x  \log(1/x)$ and $\psi^{-1}(x)\cong  \frac x  {c\log(1/x)}$ as $x\to 0^+$, where $ c =2\log^22 \int_0^\infty\frac {1-\cos r} {r^2 }dr= \pi\log^22$.    Therefore  $\psi$ is regularly varying at $0$ with index $1$.  Hence
$$ L(u)=\frac{1}{\pi}\int^\infty_{\psi^{-1}(u)}\frac{dr}{\psi(r)}\cong  \frac{1}{c\pi} \log\log (1/u),\quad u\to 0^+$$
and from Corollary \ref{0hit_a} we infer  that
$$ \p^x(T_0>t)\cong  (\pi\log2)^2 \frac{ K(x)}{\log \log t},\quad t\to \infty.$$

 \end{exmp}

\section{Behaviour of harmonic functions}

This section prepares some tools used in the  sequel  for estimating the tail function for the hitting time of an interval. On the other hand the results are interesting on their own.
In the first subsection we prove the global Harnack  inequality under global  weak scaling assumption for $\psi$, while in the second we provide  some boundary type estimates  for harmonic functions.

We start with a lemma which shows a very useful  property of the compensated kernel.
\begin{lem}\label{3K}
Let $0<R\le \infty$.  For $x>0$ we have  $\E^x K(X_{\tau_{(0, R)}})\leq  K(x).$
\end{lem}
\begin{proof} We provide the proof for $R=\infty$, since the case $R<\infty$ is similar. By \cite[Theorem 1.1]{MR2603019} we know that the function $K(x)$ is invariant for the killed process upon hitting $\{0\}$ which implies its harmonicity on $\{0\}^c$.
Let $R>1  $ such that $1/R <x<R$. Then by harmonicity $\E^x K(X_{\tau_{(1/R, R)}})=  K(x)$. Since $\tau_{(1/R, R)}\uparrow \tau_{(0, \infty)}, R\uparrow \infty$ and $\tau_{(0,\infty)}<\infty$ almost surely, the conclusion
follows by continuity of $K$, quasi-left continuity of $X$ and Fatou's lemma.


\end{proof}

\subsection{Harnack inequality}

We say that the global Harnack inequality holds if there is a constant $C_H$ such that  for every $R>0 $ and any non-negative harmonic function on $(-R, R)$ we have
\begin{equation*}\label{HI1}\sup_{x\in (-R/2,R/2)}h(x)\leq C_H\inf_{x\in (-R/2,R/2)}h(x).\end{equation*}
Here we prove  that $\psi\in\WLSC{\alpha}{\gamma}$, $\alpha>1$ is  a sufficient condition. The Harnack inequality will be very important in the next subsection to find the boundary behaviour of certain harmonic functions.

\begin{lem}\label{Green_interval}Let  $\psi\in\WLSC{\alpha}{\gamma}$, $\alpha>1$. Then there are  $\lambda_1, \lambda_2 $ depending on the scalings such that for any $R>0$, 
$$G_{(-R,R)}(x,y)\geq \lambda_2 K(R),\quad |x|,|y|\leq \lambda_1 R.$$
\end{lem}
\begin{proof}
Since $\psi\in\WLSC{\alpha}{\gamma}$, $\alpha>1$, then $ K$ has the lower scaling property with index $\alpha-1$. Hence there exists $\delta <1/2$, depending on the scalings, such that $$\inf_{z\ge R(1-\delta)}K(z)- \sup _{|z|\le\delta R}K(z)\ge (1/2)\inf_{z\ge R(1-\delta)}K(z)\ge (1/2)\inf_{z\ge R/2}K(z).$$
Let $\lambda=\frac{\delta}{2}$ and $|x|,\,|y|\leq \lambda R$. By Lemma \ref{3K} we have $\E^{x+R}K(X_{\tau_{(0,2R)}})\leq K(x+R)$.
Hence
\begin{eqnarray*}
G_{(-R,R)}(x,y)&=&G_{(0,2R)}(x+R,y+R)=\Go(x+R,y+R)-\E^{x+R}\Go(X_{\tau_{(0,2R)}},y+R)\\&=&K(x+R)-K(y-x)-\E^{x+R} K(X_{\tau_{(0,2R)}})+\E^{x+R} K(X_{\tau_{(0,2R)}}-y-R)\\
&\ge &   \E^{x+R} K(X_{\tau_{(0,2R)}}-y-R)-  K(y-x)  \\
&\geq&\inf_{z\ge R(1-\lambda)}K(z)- \sup _{|z|\le2\lambda R}K(z)\\
&\geq& (1/2)\inf_{z\ge R/2}K(z)\geq \lambda_2 K(R),
\end{eqnarray*}
where $\lambda_2$ depends only on the scalings.
\end{proof}

\begin{prop}\label{PropKrylovSafonov}Let  $\psi\in\mathrm{WLSC}(\alpha,\gamma)$, $\alpha>1$.  There exists a constant $\delta\le \lambda_1$ dependent only on the scalings  such that
for any $R>0$ and any non-empty Borel $A\subset (-\delta R,\delta R)$,
$$\p^x(T_{A}<\tau_{(-R,R)})\geq \frac 12,\qquad |x|\leq \delta R.$$
\end{prop}
\begin{proof}
Let $ |a|\le R/4$.
Let $D=(-R/2,0)\cup (0,R/2)$. By \eqref{B1} and then by Proposition \ref{exittime1}, for $|x-a|\le R/4$,
$$\p^x(T_{a}>\tau_{(-R,R)})\le \p^{x-a}(T_{0}>\tau_{(-R/2,R/2)})\leq C_2 \frac {\E^{x-a} \tau_D} {V^2(R/2)}\le 8C_2  K(x-a)\frac  R{V^2(R)}. $$
Since $\frac  {V^2(R)}R\approx K(R)$, with comparability constant dependent on the scalings,  then
$$\p^x(T_{a}>\tau_{(-R,R)})\le c  \frac {K(x-a)} {K(R)},\quad |x-a|<R/4,$$ with $c$ dependent on the scalings. Next,  we can use WLSC property for $K$ with index $\alpha-1$ to  choose $\delta<1/2$ (dependent only on the scalings)  small enough, such that
$$ \p^x(T_{a}>\tau_{(-R,R)})\le 1/2,\quad |x-a|<2\delta R.$$
Let $x\in A\subset (-\delta R,\delta R)$ and $a\in A$. Then

$$\p^x(T_{A}>\tau_{(-R,R)})\le \p^x(T_{a}>\tau_{(-R,R)})\le 1/2.$$
%
\end{proof}

Let $R_0=\delta R$, with $\delta$ chosen in the preceding proposition.
\begin{prop}\label{PropSuppportOut} Let $\psi\in\WLSC{\alpha}{\gamma}$, $\alpha>1$.
 Then for any $R>0$, and any non-negative function $F$ such that $(\mathrm{supp} F)^c \subset (-R,R)$,
$$\E^xF(X_{\tau_{(-R_0,R_0)}})\leq \frac2{\lambda_2} \E^yF(X_{\tau_{(-R,R)}}),\qquad |x|,\,|y|<R_0.$$
\end{prop}
\begin{proof}
Denote $\nu(w,A)= \nu(A-w), w\in \R$ and Borel $A\subset\R$.
By the Ikeda-Watanabe formula and Lemma \ref{Green_interval},
\begin{eqnarray*}\label{Supp2} \E^yF(X_{\tau_{(-R,R)}}) &\geq& \int_{(-R,R)^c}\int_{-R_0}^{R_0} F(z) G_{(-R,R)}(y,w)\nu(w,dz)dw\\&\geq& \lambda_2 K(R) \int_{(-R,R)^c}\int_{-R_0}^{R_0} F(z)\nu(w,dz)dw.\end{eqnarray*}
Again, by  the Ikeda-Watanabe formula, subadditivity of $K$ and Proposition \ref{Greenb},
\begin{eqnarray*}
\E^xF(X_{\tau_{(-R_0,R_0)}})&\leq&\int_{(-R,R)^c}\int_{-R_0}^{R_0}F(z)\Go(x+R_0,w+R_0)\nu(w,dz)dw\\ &\leq& 2K(R_0)\int_{(-R,R)^c}\int_{-R_0}^{R_0}F(z)\nu(w,dz)dw.
\end{eqnarray*}
Hence
$$\E^xF(X_{\tau_{(-R_0,R_0)}})\leq \frac2{\lambda_2} \E^yF(X_{\tau_{(-R,R)}}).$$
\end{proof}

\begin{thm}\label{Harnack} If $\psi\in\WLSC{\alpha}{\gamma}$, $\alpha>1$,
then the global scale invariant Harnack inequality holds. That is there is a constant $C_H$ dependent only the scalings such that for any $R>0$ and any non-negative harmonic function on $(-R,R)$  we have
\begin{equation}\label{HI1}\sup_{x\in (-R/2,R/2)}h(x)\leq C_H\inf_{x\in (-R/2,R/2)}h(x).\end{equation}

\end{thm}
\begin{proof}
We prove the result for bounded harmonic functions. The boundedness assumption   can be removed in a  similar way as in \cite[Theorem 2.4]{MR2031452}. 

  With  Propositions \ref{PropSuppportOut}  and \ref{PropKrylovSafonov} at hand   we can use the approach of Bass and Levin (\cite{MR1918242})   to get the existence of constants $c_1=c_1(\alpha,\gamma)$ and $a=a(\alpha,\gamma)<1$ such that, for  any function $h$ non-negative and bounded on $\R$ and  harmonic in a ball $(-R,R)$, $R>0$,
$$\sup_{x\in (-aR,aR)}h(x)\leq c_1\inf_{x\in (-aR,aR)}h(x).$$
Next, we use the standard chain argument to get
 $$\sup_{x\in (-R/2,R/2)}h(x)\leq C_H\inf_{x\in (-R/2,R/2)}h(x),$$
 where $C_H=C_H(c_1,a)$.
\end{proof}

\subsection{Boundary behaviour}


In this subsection we prove certain estimates of non-negative functions which are harmonic on $(0, R), 0<R\le \infty$. We show that under appropriate assumptions the function $V(x)$ provides the right order of decay at the boundary at $0$ for harmonic functions we consider. The obtained results are then  used in Section \ref{Interval} to estimate the tail function of the hitting time of an interval.

In our development the following {Property} $\textbf{(H)}$ of the derivative  of $V$ is crucial.  Below, in Remark \ref{Hproperty}, we discus the situations when it holds. We also mention that we do not know any example of a symmetric L\'evy process {with an unbounded symbol} for which the property is not satisfied.

{\it Property} $\textbf{(H)}$.

We say that $X$ satisfies $\textbf{(H)}$ if there is a constant $H\ge 1$ such that for any $0<\delta\le w\le u \le w+2\delta$ we have
$$V^\prime(u)\le H  V^\prime(w).$$
\begin{rem}\label{Hproperty} The assumption $\textbf{(H)}$ is satisfied in the following situations:

a)  $\psi\in\WLSC{\alpha}{\gamma}$, $\alpha>1$. The constant $H$ depends only on the scalings. Since $V^\prime$ is harmonic on $(0, \infty)$ this follows from  Theorem \ref{Harnack}.

b) $X$ is a subordinate Brownian motion and $\psi\in\WLSC{\beta}{\gamma}$, $\beta>0$. The constant $H$ depends only on the scalings. This follows from \cite[Theorem 7]{MR3225805}.

c) $X$ is a special subordinate Brownian motion, since in this case $V^\prime$ is non-increasing \cite[Lemma 7.5]{BGR3}.

\end{rem}

\begin{lem}\label{GreenH} Suppose that $\textbf{\textrm{(H)}}$ holds. Then for $0<x< \delta<y/3$  we have

$$G_{(0,\infty)}(x,y)\le H^2 G_{(0,\infty)}(2\delta,y)\frac {V(x)}{V(\delta)}.$$

\end{lem}

\begin{proof} Let $0\le u \le x$. Since $x<y$, and  $\delta \le  y-x \le   y-x + u\le  y-x+\delta $ by \eqref{Gformula_F} and the property $\textbf{(H)}$ we have
$$G_{(0,\infty)}(x,y)=\int^{x}_0V'(u)V'( y-x + u)du\le H V'( y-x )V(x).$$
Next, since    $ \delta < y-2\delta +u \le    y-x \le y-2\delta +u +2\delta$ for $0\le u\le \delta $, using again the property $\textbf{(H)}$ we arrive at
$$V'( y-x )\le H V'(y-2\delta +u).$$
Multiplying both sides by $ V'( u )$ and integrating  over $[0, \delta]$ we obtain
$$V'( y-x )V(\delta)\le H \int^{\delta }_0 V'(y-2\delta +u)V'( u )du\le H \int^{2\delta }_0 V'(y-2\delta +u)V'( u )du= H G_{(0,\infty)}(2\delta,y), $$
which completes the proof.
\end{proof}

 According to \cite[Theorem VI.20]{MR1406564} we have for any non-negative function $f$ on $[0,\infty)$,
\begin{equation*}\label{Gformula}\int^{\infty }_0 f(y)G_{(0,\infty)}(x,y)dy= \int^{\infty }_0V'(y)dy\int^{x }_0f(x+y-u)V'(u)du.\end{equation*} 
Applying this to the indicator of the interval $[0, z],\, z>0,$  
 we have
$M(x, z)= \int^{z}_0G_{(0,\infty)}(x,y)dy= \int^{z}_0V'(y)dy\int^{x}_{(x+y-z)\vee0}V'(u)du.$ It is then  clear that   \begin{equation}\label{Green} V(z-x)V(x)\le M(x, z)\le  V(x)V(z),\quad 0<x\le z<\infty.  \end{equation}

\begin{lem} \label{H_cor}
Suppose that  $\textbf{(H)}$ holds. Let F(z) be non-negative subadditive on $\R$ and $\E^{x}F(X_{\tau_{(0,\infty)}})\le   F(x)$,  $x>0$. Then, for $0<x<1$,
$$\E^{x}[X_{\tau_{(0,\infty)}}\leq-2;F(X_{\tau_{(0,\infty)}})]\le cH^2 F^*(1) \frac{V(x)}{V(1)},$$
where $F^*(x)= \sup_{|z|\le |x|}F(z).$ The constant $c$ is absolute.
\end{lem}

\begin{proof}
By the Ikeda-Watanabe formula, 
\begin{eqnarray*}
\E^{x}[X_{\tau_{(0,\infty)}}\leq-2;F(X_{\tau_{(0,\infty)}})]
&=&\int^{-2}_{-\infty}F(z)\int^\infty_0G_{(0,\infty)}(x,y)\nu(z, dy)dz.
\end{eqnarray*}
By Lemma \ref{GreenH}, 
$$G_{(0,\infty)}(x,y)\le H^2\frac{V(x)}{V(2)}G_{(0,\infty)}(4,y),\quad  x\le 2<6\le y.$$
Hence,
\begin{eqnarray*}
\int^\infty_{6}G_{(0,\infty)}(x,y)\nu(z, dy)
&\leq&H^2\frac{V(x)}{V(2)}\int^\infty_{6}G_{(0,\infty)}(4,y)\nu(z, dy).
\end{eqnarray*}
Note that $\E^z F(X_{\tau_{(0,\infty)}})\le F(z)$, which implies 
$$I=\int^{-2}_{-\infty}F(z)\int^\infty_{6}G_{(0,\infty)}(x,y)\nu(z, dy)dz\leq H^2\frac{V(x)}{V(2)}\E^{4}F(X_{\tau_{(0,\infty)}})\leq H^2\frac{V(x)}{V(2)}F(4).$$
 Observe that by subadditivity of $F$, $ F(w+y)\le  F(w)+ F(y)\le F(w)+ F^*(6)$ if $0<y<6$ and $w<-2-y$. By \eqref{Green} we have  $\int^{6}_{0}G_{(0,\infty)}(x,y)dy\le V(6)V(x)\le 2V(4)V(x)$, hence
\begin{eqnarray*}
II&=&\int^{-2}_{-\infty}F(z)\int_0^{6}G_{(0,\infty)}(x,y)\nu(z, dy)dz=  \int_0^{6}G_{(0,\infty)}(x,y) \int^{-2-y}_{-\infty}F(w+y)\nu(dw)dy\\
&\leq&2V(x)\frac{V(2)V(4)}{V(2)}\int^{-2}_{-\infty}F(w)\nu(dw)+ 2V(x)F^*(6)V(4)\nu([2,\infty)). 
\end{eqnarray*}
Note that by \eqref{Green},  $V(2)V(4)\le  \int_0^{6}G_{(0,\infty)}(2,y)dy,$ and $F(w-y) \le  F(w)+F^*(6)$ if $0<y<6$ and $w<-6$ 
which imply
\begin{eqnarray*}{V(2)V(4)}\int^{-12}_{-\infty}F(w)\nu(dw)&\le& \int_0^{6} \int^{-12}_{-\infty}F(w)G_{(0,\infty)}(2,y)\nu(dw)dy\\&\le&  \int_0^{6} \int^{-6-y}_{-\infty}F(w)G_{(0,\infty)}(2,y)\nu(dw)dy\\&=&  \int_0^{6} \int^{-6}_{-\infty}F(w-y)G_{(0,\infty)}(2,y)\nu(y, dw)dy\\&\le& \int_0^{6} \int^{-6}_{-\infty}(F(w)+F^*(6))G_{(0,\infty)}(2,y)\nu(y, dw)dy\\&\le&  \E^{2} F(X_{\tau_{(0, \infty)}})+F^*(6)\p^{2}(X_{\tau_{(0, \infty)}}\le -6)\\&\le& F(2)+F^*(6).\end{eqnarray*}  
Next, by  Lemma \ref{upper_den},   $$\int_{-12}^{-2}F(w)\nu(dw)\le F^*(12)\nu[2,\infty)\le C_2 \frac {F^*(12)}{V^2(2)}.$$

Combining all the estimates obtained above and using subaddativity of $F^*$ and $V$   we conclude that there is an absolute constant $c$ such that  $$\E^{x}[X_{\tau_{(0,\infty)}}\leq-2;F(X_{\tau_{(0,\infty)}})]= I+II\le cH^2 F^*(1) \frac{V(x)}{V(1)},\quad {0<x<1},$$
which   ends the proof.

%
\end{proof}

\begin{lem}\label{BHPLowerGeneral}

{Let $\psi\in\WLSC{\alpha}{\gamma}$, $\alpha>1$} and
let $F$ be a non-negative harmonic function on $(0,2R), R>0$.  Suppose that $ r>0 $ is such that  $V(R)\geq 2V(r)/C_4$, where $C_4$ is the constant from Lemma \ref{exit}. Then for $0<x<r$,
$$\frac{F(x)}{F(r)}\geq \frac {C_4}2 \ (C_H)^{R/r+1} \frac{V(x)}{V(r)}, $$ where $C_H$ is the constant from the Harnack inequality \eqref{HI1}, which depends  only on the scalings.
%
\end{lem}
\begin{proof}  Since $ F$ is harmonic then using the Harnack inequality (Theorem \ref{Harnack}) we have for every $r\le x,y\le R$ such that  $|x-y|<r$,
$F(x)\ge C_HF(y)$. By the chaining argument we have for any $r\le x\le R$,
\begin{equation} \label{Harnack1}F(x)\ge (C_H)^{R/r+1}F(r).\end{equation}
   By Lemma \ref{exit}, \begin{eqnarray*}\p^x(X_{\tau_{(0,r)}}\in[r,R))&=& \p^x(X_{\tau_{(0,r)}}\ge r) -  \p^x(X_{\tau_{(0,r)}}\ge R)
\ge \p^x(X_{\tau_{(0,r)}}\ge r) -  \p^x(X_{\tau_{(0,R)}}\ge R)\\&\ge& C_4 \frac{V(x)}{V(r)}- \frac{V(x)}{V(R)}\ge \frac {C_4}2 \frac{V(x)}{V(r)}. \end{eqnarray*}
Note that by \eqref{Harnack1}, quasi left-continuity of $X$ and harmonicity of $F$,
\begin{eqnarray*}F(x)&=&\lim_{\varepsilon\to0^+}\E^xF(X_{\tau_{(\varepsilon,r)}})\geq \lim_{\varepsilon\to0^+}\E^x\[F(X_{\tau_{(\varepsilon,r)}}),X_{\tau_{(\varepsilon,r)}}\in[r,R]\]\\
&\geq&C_H^{R/r+1} F(r)\lim_{\varepsilon\to0^+} \p^x(X_{\tau_{(\varepsilon,r)}}\in[r,R])=C_H^{R/r+1} F(r) \p^x(X_{\tau_{(0,r)}}\in[r,R])\\&\geq& \frac{C_H^{R/r+1}C_4}{2}F(r)\frac{V(x)}{V(r)}.
\end{eqnarray*}
\end{proof}


%

\section{Hitting times of intervals}\label{Interval}

Throughout this section  $B_R= [-R,R], \ R>0$. The goal is to find sharp estimates for the tail function of $T_{B_R}$ and  we start with the case $R=1$. Once this is done we use the scaling argument to treat any $R>0$. The proposition below provides an effective tool for the upper bound.

\begin{prop}\label{upperB} Suppose that the condition $\textbf{(H)}$  holds. Then
  $$\p^x(T_{B_1}>1)\leq  cH^2\frac{V(|x|-1)} {V(|x|)}\left[\sup_{|z|\le 1}\p^{z}(T_{0}>1/2)+ \p^{x-1}(T_{0}>1/2)\right] , \quad 
 1> 1/\psi^*(1),\, |x|>1.$$
 The constant $c$ is absolute.

\end{prop}
\begin{proof} If $|x| \ge2$ we have, by subadditivity of $V$,  $\frac{V(|x|-1)}{V(|x|)}\ge \frac13$, hence the conclusion is obvious.

Let  $1<x<2$. The condition    $1> 1/\psi^*(1)$, by Lemma  \ref{ch1V}, implies $1/V(1)\le C_1$.  Then,  by Lemma \ref{halfspace} and subadditivity of $V$, $$\p^x(\tau_{(1,\infty)}>1/2)\le 2\sqrt2  \frac{C_1}{{C_3}}\frac{V(x-1)}{V(1)}\p^2(\tau_{(1,\infty)}>1)\le
{4}\sqrt2 \frac{C_1}{{C_3}}\frac{V(x-1)}{V({x})}{\p^1(T_{0}>1/2)}.$$ Since
$$\p^x(T_{B_1}>1)\le  \p^x(\tau_{(1,\infty)}>1/2)+   \E^x\p^{X_{\tau_{(1,\infty)}}}(T_{B_1}>1/2)$$
 it is enough to estimate the harmonic function
 \begin{eqnarray*}
\E^x\p^{X_{\tau_{(1,\infty)}}}(T_{B_1}>1/2)&\le& \E^x[X_{\tau_{(1,\infty)}}\leq-1; \p^{X_{\tau_{(1,\infty)}}}(T_{1}>1/2)]\\&=&\E^{x-1}[X_{\tau_{(0,\infty)}}\leq-2;\p^{X_{\tau_{(0,\infty)}}}(T_{0}>1/2)].
\end{eqnarray*}
 Let $F(z)= \p^{z}(T_{0}>1/2)$. Observe that this function is subadditive and satisfies the assumptions of Lemma \ref{H_cor}.  Therefore the conclusion follows from Lemma \ref{H_cor}.

\end{proof}


\begin{cor} \label{ball22} Let   $\psi\in\WLSC{\alpha}{\gamma}$, $\alpha>1$. If   $|x|>1$ and $ 1> 1/\psi^*(1)$ then
$$\p^x(T_{B_1}>1)\le  c\frac{V(|x|-1) K(|x|)} {V(|x|)  K(1/\psi^{-1}(1))}\wedge 1 \approx  \frac{V(|x|-1) K(|x|)} {V(|x|) \psi^{-1}(1)}\wedge 1.$$
The constant $c$ depends only on the scalings.



\end{cor}

\begin{proof}

By Remark \ref{Hproperty} we find a constant $H$ dependent only on the scalings such that the property $\textbf{(H)}$ holds. Therefore, by applying Propositions \ref{upperB} and \ref{0hit} together with Lemma \ref{KWLSC} we end the proof.

\end{proof}

\begin{rem}\label{ext}If  $ \p^x(T_0>t)\le  c\(\frac{K(x)}{K(1/\psi^{-1}(1/t))}\wedge 1\)$ the WLSC assumption  is merely to assure the property $\textbf{(H)}$. However there are many examples for which $V^\prime$ is  non-increasing and then this property holds automatically with the constant $H=1$. For example if $X$ is a special subordinate Brownian motion satisfying \eqref{0reg}, then the estimate from the preceding corollary holds with an absolute constant. In particular   $\psi(x)=|x|+|x|^2$ defines a special subordinate Brownian motion and it does not have the lower scaling property with index $\alpha >1$.\end{rem}

Next, we deal with the lower bound.
\begin{prop} \label{ball} Let
$\psi\in\WLSC{\alpha}{\gamma}$, $\alpha>1$, and  let  $ 1>1/\psi^*(1)$.  There is $x^*\ge 2$, which depends only on the scaling characteristics, such that    for $|x|\ge x^*$ we have
%

$$\p^x(T_{B_1}>2)\ge \Ce   \left(\frac{K(|x|)}{K(1/\psi^{-1}(1))}\wedge 1\right)\approx \left(\frac{K(|x|)}{\psi^{-1}(1)}\wedge 1\right) .$$
The constant $\Ce $ depends only on the scalings.
\end{prop}
\begin{proof}

%
By symmetry we may assume that $x>0$.
Let $f(t)= \p^x(T_{B_1}>t),\, f_0(t)= \p^x(T_0>t)$.
We begin with a simple observation relating the Laplace transforms of $f(t)$ and  $f_0(t)$. {By \eqref{laplaceComp}},
\begin{eqnarray*}
\lambda u^\lambda(0)\mathcal{L}f(\lambda)&=&K^\lambda(x) -\E^xe^{-\lambda T_{B_1}}K^\lambda(X_{T_B})\\
&\ge&K^\lambda(x) -\E^xK(X_{T_{B_1}})\\
&=& \lambda u^\lambda(0)\mathcal{L}f_0(\lambda)-\E^xK(X_{T_{B_1}}).
\end{eqnarray*}
Let ${\gamma}(\theta, z) = \int_0^zu^{\theta-1}e^{-u}du, z>0, \theta>0 $ be the lower incomplete Gamma function of index $\theta$.
We pick $0<b\leq1$ to be specified later. By  \cite[Lemma 5]{MR3165234}, ${\gamma}(1, 1)f_0(s)\le \mathcal{L}f_0(s^{-1})s^{-1}, \ s>0$. Moreover,  \eqref{WUSC0} implies  $\mathcal{L}f_0(s^{-1})\le c_1\(\lambda s\)^{1/\alpha}\mathcal{L}f_0(
{\lambda }),\ {s\le \lambda^{-1}}, $ where $c_1$ depends on the scalings. Hence 
\begin{align*}
\lambda \mathcal{L}f(\lambda)&=\lambda\int_0^{b\lambda^{-1}}e^{-\lambda s}f(s)ds+\lambda\int_{b\lambda^{-1}}^\infty e^{-\lambda s}f(s)ds\\
&\leq \frac{\lambda}{\gamma(1,1)}\int^{b\lambda^{-1}}_0e^{-
{\lambda s}}\mathcal{L}f_0(s^{-1})s^{-1}ds+f(
b\lambda^{-1})\int_{b\lambda^{-1}}^\infty e^{-\lambda s}\lambda ds\\
&\leq \frac{\lambda}{\gamma(1,1)}\int^{b\lambda^{-1}}_0 c_1\(\lambda s\)^{1/\alpha}\mathcal{L}f_0(
{\lambda })e^{-\lambda s}s^{-1}ds+f(b\lambda^{-1})e^{-b}\\
&= c_1\frac{\gamma(1/\alpha,b)}{\gamma(1,1)}\lambda\mathcal{L}f_0(
{\lambda})+f(b\lambda^{-1})e^{-b}.
\end{align*}
Next we choose the largest $b\le 1$ such that
$2c_1\gamma(1/\alpha,b)
\le {\gamma(1,1)}=1-e^{-1}$.
Since  $\lambda \mathcal{L}f(\lambda)\ge \lambda \mathcal{L}f_0(\lambda)- \frac{\E^xK(X_{T_B})}{u^\lambda(0)}$, then 
$$
f(b\lambda^{-1})\geq  \lambda\mathcal{L}f_0(\lambda){/2}- \frac{\E^xK(X_{T_B})}{u^\lambda(0)}= \frac {K^\lambda(x)}{2u^\lambda(0)}-\frac{\E^xK(X_{T_B})}{u^\lambda(0)}.$$
 If $x\psi^{-1}(\lambda)\le 1$ then  by Lemma  \ref{potEst},  $K^\lambda(x)\ge c_2 K(x)$, hence in this case
 $$f(b\lambda^{-1})\geq   \frac{c_2{K(x)}-{\E^xK(X_{T_B})}}{u^\lambda(0)}.$$
 Letting $\lambda =1/t$ and applying again  Lemma  \ref{potEst} to estimate $u^\lambda(0)$ we have for $t\ge \frac 1{\psi^*(1/x)}$,
  $$f(bt)\geq  c_3 \left[\frac{c_2 {K(x)}-{\sup_{|z|\le 1}K(z)}}{K(1/\psi^{-1}(1/t))}\right].$$
By Lemma \ref{KWLSC}  the lower scaling property with index $\alpha-1$ holds for $K$, therefore  we can find $x^*\ge 2$, dependent on the scalings, such that  $c_2 {K(x)}-{\sup_{|z|\le 1}K(z)}\ge \frac  {c_2}2 K(x), \quad x\ge x^*$. Hence, for  $x\ge x^*$ and $t\ge \frac 1{\psi^*(1/x)}$ we have
$$f(bt)\geq  c_4 \frac{K(x)}{K(1/\psi^{-1}(1/t))}.$$
For  $x\ge x^*\ge 2$ and $ t\le \frac 1{\psi^*(1/x)}$ we apply subaddativity of $V$ and Lemma \ref{ch1V}  to get $\frac{V(x-1)}{\sqrt{t}}\ge \frac12 \frac{V(x)}{\sqrt{t}}\ge \frac1{2C_1} $.  Next, applying  Lemma \ref{halfspace}  to arrive at

$$f(bt)\geq \p^x(\tau_{(1,\infty)}\ge bt) \ge  C_3\left( \frac{V(x-1)}{\sqrt{bt}}\wedge 1\right)\ge C_3\left( \frac{1}{2C_1\sqrt{b}}\wedge 1\right). $$
Therefore we have proved that for $x\ge x^*$ and any $t>0$ we have
$$f(bt)\geq  c_5\min\left\{ \frac{K(x)}{K(1/\psi^{-1}(1/t))}, 1\right\},$$
where $c_5$ depends on the scalings.
In particular taking $t=2/b$ we obtain
$$f(2)\geq  c_5\min\left\{ \frac{K(x)}{K(1/\psi^{-1}(b/2))}, 1\right\}\ge c_6\min\left\{ \frac{K(x)}{K(1/\psi^{-1}(1))}, 1\right\},\quad |x|\ge x^*, $$
where the last inequality follows from scaling property for $K$ and $\psi^{-1}$ {(see Lemma \ref{KWLSC} and \eqref{psiInvScal})}. The constant $c_6$ depends on the scalings.






\end{proof}


\begin{lem} \label{ball20} Let  $\psi\in\WLSC{\alpha}{\gamma}$, $\alpha>1$. If   $1<|x|<x^*$ and $ 1> 1/\psi^*(1)$, where $x^*$ is chosen in the preceding lemma,  then
$$\p^x(T_{B_1}>1)\ge  c\frac{V(|x|-1) K(|x|)} {V(|x|) {K(1/\psi^{-1}(1))}}\wedge 1  .$$
The constant $c$ depends only on the scalings.


\end{lem}

\begin{proof}
We may and do assume that $1\le x<x^*$. By the strong Markov property
we have for any $z\ge 1$,
$$ \p^z(T_{B_1}>2)\le  \p^z(\tau_{(1,\infty)}>1)+   \E^z\p^{X_{\tau_{(1,\infty)}}}(T_{B_1}>1)$$ and
$$ \p^z(T_{B_1}>1)\ge  \frac12\p^z(\tau_{(1,\infty)}>1)+  \frac12\E^z\p^{X_{\tau_{(1,\infty)}}}(T_{B_1}>1).$$
Using Lemma \ref{BHPLowerGeneral} we estimate the harmonic function $F(z)= \E^z\p^{X_{\tau_{(1,\infty)}}}(T_{B_1}>1)$,

 $$F(x)\ge c_1\frac{V(x-1)}{V(2)} F(x^*),$$ with the constant $c_1$ dependent only on the scalings.
From Lemma \ref{halfspace} and subaddativity of $V$ we infer that  $\p^x(\tau_{(1,\infty)}>1)\ge c_2  \frac{V(x-1)}{V(2)}\p^{x^*}(\tau_{(1,\infty)}>1)$, with $c_2$ dependent only on the scalings.   Hence,
\begin{eqnarray*}\p^x(T_{B_1}>1)&\ge& 1/2\left( \p^x(\tau_{(1,\infty)}>1)+   \E^x\p^{X_{\tau_{(1,\infty)}}}(T_{B_1}>1)\right)\\&\ge& \frac{c_1\wedge c_2}2  \frac{V(x-1)}{V(2)} \left( \p^{x^*}(\tau_{(1,\infty)}>1)+   \E^{x^*}\p^{X_{\tau_{(1,\infty)}}}(T_{B_1}>1)\right)\\&\ge& \frac{c_1\wedge c_2}2   \frac{V(x-1)}{V(2)} \p^{x^*}(T_{B_1}>2).\end{eqnarray*}
Applying Proposition \ref{ball} we get  $\p^{x^*}(T_{B_1}>2) \ge \Ce  \(\frac {K(2)}{K(1/\psi^{-1}(1))}\wedge1\)$, 
 which completes the proof.
\end{proof}




Now we are ready to state and prove the main result of this section.
\begin{thm} \label{ball2}  Let $\psi\in\WLSC{\alpha}{\gamma}$, $\alpha>1$. If  $B_R=[-R,R]$ and  $|x|>R$,
 $$\p^x(T_{B_R}>t)\approx  \frac{V(|x|-R)} {\sqrt{t}\wedge V(R)}\wedge 1, \quad t\le 1/\psi^*(1/R)$$
and
$$\p^x(T_{B_R}>t)    \approx  \frac{V(|x|-R) K(|x|)} {V(|x|) K(1/\psi^{-1}(1/t))}\wedge 1   \approx  \frac{V(|x|-R) K(|x|)} {V(|x|) t\psi^{-1}(1/t)}\wedge 1, \quad t> 1/\psi^*(1/R)  .$$
The comparability constants depend only on the scaling characteristics.



\end{thm}
\begin{proof}

If $t\leq 1/\psi^*(1/R)$ the estimates hold by \cite[Remark 6]{BGR3} and Lemma \ref{halfspace}.

 Let $t>0, R>0$ be fixed.  We consider a space and time rescaled process $Y_s=X_{ts}/R, s\ge 0$.
  Let $K^t_R$, etc. be objects corresponding to the process $Y$. Then
 \begin{eqnarray*}\psi^t_R(x)&=&  t\psi(x/R),\\
K^t_R(x)&=& \frac Rt K(xR),\\
V^t_R(x)&=&  \frac{V(xR)}{\sqrt{t}},\\
\psi^{-1}_R(x)&=&  R\psi^{-1}( x/t).\end{eqnarray*}
Let $T^{Y}_{B_1}$ be the hitting time of $B_1$ by the process $Y$. Observe that $\psi^t_R(x)$ has exactly the same scaling property (with the same scaling characteristics)  as $\psi(x)$.
Let  $ t>  1/\psi^*(1/R)$ or equivalently  $1>  1/(\psi^t_R)^*(1)$. We now apply 
Corollary \ref{ball22},  Proposition \ref{ball} and Lemma \ref{ball20} to get
\begin{eqnarray*}\p^x(T_{B_R}>t)&=& \p^{x/R}(T^{Y}_{B_1}>1)\\&\approx& \frac{V^t_R(|x/R|-1) K^t_R(|x/R|)} {V^t_R(|x/R|)  K^t_R\left(1/(\psi^t_R)^{-1}(1)\right)}\wedge 1\\
&=& \frac{V(|x|-R) K(|x|)} {V(|x|) K(1/\psi^{-1}(1/t))}\wedge 1 \\  &\approx&  \frac{V(|x|-R) K(|x|)} {V(|x|) t\psi^{-1}(1/t)}\wedge 1
,\end{eqnarray*}
where the comparability constants depend only on the scalings of $\psi$.
%
 \end{proof}

\begin{cor} Assume that the second moment of $X$ is finite and $\psi\in\WLSC{\alpha}{\gamma}$, $\alpha>1$.  Let $R>1$ and  $B_R=[-R,R]$.  Then

$$\p^x(T_{B_R}>t)\approx  \frac{V(|x|-R)} {\sqrt{t}}\wedge 1, \quad t> 0.$$
The comparability constant depends on $\psi$.
\end{cor}

\begin{proof} If the second moment is finite, then $\psi(x)\approx x^2,  |x|<1$, hence $\psi^{-1}(x)\approx \sqrt{x},0\leq  x<1$. Moreover $ K(w)\approx V(w)\approx w, w>1$. Hence
we obtain the conclusion applying Theorem \ref{ball2}.
\end{proof}

\begin{rem}\label{ext1}    If $X$ is a special subordinate Brownian motion satisfying \eqref{0reg}, then the upper bound from Theorem \ref{ball2} is true without the assumption $\psi\in\WLSC{\alpha}{\gamma}$, $\alpha>1$. This follows from the fact that property $\textbf{(H)}$ holds for such processes (see Remark \ref{Hproperty}). In particular   $\psi(x)=|x|+|x|^2$ defines a special subordinate Brownian and it does not have a lower scaling property with $\alpha >1$ but we have
$$\p^x(T_{B_R}>t)\le c  \frac{V(|x|-R)} {\sqrt{t}\wedge V(R)}\wedge 1, \quad t\le 1/\psi^*(1/R), $$
and
$$\p^x(T_{B_R}>t)\le c   \frac{V(|x|-R) K(|x|)} {V(|x|) K( 1/\psi^{-1}(1/t))}\wedge 1, \quad t> 1/\psi^*(1/R)  .$$
Here the constant $c$ is independent of $R$, $V(x)\approx   \sqrt{x}\wedge x,\ x \ge 0$ and $K(x) \approx \log (1+|x|),\ x\in \R$.

By inspecting the proof of Proposition \ref{ball}
it is clear that we can prove a  lower bound
$$\p^x(T_{B_R}>t)\ge c_R   \frac{V(|x|-R) K(|x|)} {V(|x|) K( 1/\psi^{-1}(1/t))}\wedge 1, \quad t> 1/\psi^*(1/R),$$
but the constant $c_R $ will be dependent on $R$ (to choose $x^*$ as in Proposition \ref{ball} one can use unboundedness of $K$ instead of the scaling property).

 \end{rem}

\begin{exmp}Let $\nu(r)= \frac 1{|r|^3 \log^2(2+1/|r|)}, \ r\in \R$. Since $\nu$ is decreasing on $(0,\infty)$, by Lemma \ref{cos}, $\psi(x)\approx \int^\infty_0(1\wedge(xr)^2)\nu(r){dr}= x^2 \int^{1/x}_0r^2\nu(r){dr}+\int_{1/x}^\infty\nu(r){dr}$. Elementary
calculations  show that
\begin{eqnarray*}\psi(x)&\approx& x^2\frac{\log (2+1/x)}{\log (2+x)}, \quad x>0,\\
\psi^{-1}(x)&\approx& \sqrt{x} \frac{\sqrt{\log (2+x)}}{\sqrt{\log (2+1/x)}},\quad  x>0.\end{eqnarray*}
It is clear that $\psi$ has the weak lower scaling property with any $1<\alpha<2$. Hence
$$K(x)\approx\frac 1{x\psi(1/x)} \approx x \frac{\log (2+1/x)}{\log (2+x)},\quad  x>0$$
and
$$V(x)\approx\frac 1{\sqrt{\psi(1/x)}}\approx x \frac{ \sqrt{\log (2+1/x)}}{\sqrt{\log (2+x)}}, \quad x>0.$$
Hence $$\frac{K(x)}{V(x)}\approx  \frac{ \sqrt{\log (2+1/x)}}{\sqrt{\log (2+x)}}, \quad x>0.$$ Applying Theorem \ref{ball2} we obtain for $|x|>R$,
$$\p^x(T_{B_R}>t)\approx (|x|-R)\frac{ \sqrt{\log (2+1/(|x|-R))}}{\sqrt{t\log (2+(|x|-R))}}  \wedge 1, \quad t\le 1/\psi^*(1/R)$$
and for $t> 1/\psi^*(1/R)$,
$$\p^x(T_{B_1}>t)\approx (|x|-R)\frac{ \sqrt{\log (2+1/(|x|-R))}}{\sqrt{\log (2+(|x|-R))}} \frac{ \sqrt{\log (2+1/|x|)}}{\sqrt{\log (2+|x|)}}\frac{\sqrt{\log (2+t)}} {  {\sqrt{t\log (2+1/t)}}}\wedge 1, $$
since\  $t\psi^{-1}(1/t)\approx \sqrt{t \frac{{\log (2+1/t)}}{{\log (2+t)}}}$.
 \end{exmp}

\section{Heat kernel estimates}


 This section is devoted to finding sharp estimates of the heat kernel of the process $X$ killed after hitting an interval. We apply the previous results on hitting times and the estimates of the heat kernel of the free process obtained in \cite{MR3165234} under the assumption of unimodality of $X$ and  both lower and upper scaling property of $\psi$. At the end of the section we suggest  a certain extension of the main result, which allows to treat symmetric  processes which are not unimodal.


We denote  $ D_R=(-R,-1)\cup(1,R), \ R>1$.

\begin{prop}\label{exittime} Let $\psi\in\WLSC{\alpha}{\gamma}$, $\alpha>1$.
For $|x|\in (1,R)$ and $R>1$ we have
\begin{equation*}
\E^x\tau_{D_R}\le  \Cd R\frac{V(|x|-1) K(|x|)} {V(|x|) },
\end{equation*}
where the constant $\Cd$ depends on the scalings.
\end{prop}
\begin{proof} 

Let $x>1$.
By Proposition  \ref{Greenb},
\begin{eqnarray*}\E^x\tau_{D_R}&=&\int_{(-R,-1)\cup(1,R)} G_{(-R,-1)\cup(1,R)}(x,y)dy\leq
2\int^R_0 G_0(x-1,y)dy\\ &\le& 4R K( x-1),
\end{eqnarray*}
which gives the desired bound if $x>2$, since $\frac{V(x-1)} {V(x) }\ge 1/2$.

Assume that  and $1<x\le 2$. Let $s(u)=  \E^u\tau_{D_R}, u\in \R$. Then by the strong Markov property we have
$$ s(x)= \E^x\tau_{(1,R)}+ \E^{x}s(X_{\tau_{(1,R)}}).$$
%
Next, applying the above estimate and subaddativity of  $K$ we obtain
\begin{eqnarray*} \E^{x}s(X_{\tau_{(1,R)}})&\le& 4 R\E^{x}[K(|X_{\tau_{(1,R)}}|-1);X_{\tau_{(1,R)}}\le -1]\\&\le&
4 R\E^{x}[K(X_{\tau_{(1,R)}})+K(1);X_{\tau_{(1,R)}}\le -1]\\ &\le&
4 R\E^{x}[K(X_{\tau_{(1,\infty)}})+K(1);X_{\tau_{(1,\infty)}}\le -1]\\
&=&
4 R\E^{x-1}[K(X_{\tau_{(0,\infty)}})+K(1);X_{\tau_{(0,\infty)}}\le -2]\\ &\le&
c_1 R\frac{V(x-1) K^*(1)} {V(1) },
\end{eqnarray*}
where in the last step we applied Lemma \ref{3K} and Lemma \ref{H_cor}. Note that the constant $c_1$ depends only on the scalings.   
 Finally, applying \cite[Proposition 3.5]{MR3007664},  subadditivity of $V$ and the estimate $V^2(2)\le c_2  K(2)$ following from Lemma \ref{KWLSC} , we obtain
$$ \E^x\tau_{(1,R)}\le 2V(x-1)V(R)\le 2RV(x-1)V(2)\le c_2 R\frac{V(x-1) K(2)} {V(2) }. $$
%
The proof is completed by observing that, by  Lemma \ref{KWLSC},  $ K^*(1)\approx K(|x|), 1\le|x|\le2$   and by subaddativity  $V(1)\approx V(|x|), 1\le|x|\le 2$. 
\end{proof}

\begin{prop} \label{MRestimate} 

Let  $ R>2$ and  $1<|x|<R$.
If $\psi\in\WLSC{\alpha}{\gamma}$, $\alpha>1$, then
$$\p^x(|X_{\tau_{D_R}}|\ge R)
\le  \Cc \frac{V(|x|-1)}{V(|x|)} \frac{K(|x|)}{K(R)},$$
where the constant $\Cc$ depends only on the scalings.


\end{prop}

\begin{proof}
{By subaddativity of $V$ and $K$   it is enough to consider}
 $1<x\le (R\vee3)/2$. By  \eqref{B1}, and then Proposition \ref{exittime},
$$\p^x(|X_{\tau_{D_R}}|\ge R)
  \le C_2\frac{\E^x\tau_{D_R}}{ V^2(R)}\le    C_2\Cd R\frac{V(x-1) K(x)} {V(x)  V^2(R) }.$$
 %
  The proof is completed { by observing that   $K(R)\approx  V^2(R)/R$ with the comparability constant dependent only on the scalings, which folows from Lemma \ref{KWLSC}.  } \end{proof}

  The following lemma is consequence of \cite[(3.2)]{MR632968}, \cite[Corollary 1]{MR3225805} and  the weak lower scaling property.
   \begin{lem}\label{ball001} Let  $\psi\in\WLSC{\alpha}{\gamma}$, $\alpha>0$. There is a constant $\Cb\ge 1$ dependent on the scalings such that
   for  $R= \chi/\psi^{-1}(1)$,  $\chi>1$, we have
\begin{eqnarray*} \p^0(\tau_{(-R,R) }\le 1) \le  \frac{\Cb}{\chi^\alpha}  .\end{eqnarray*}
    \end{lem}
     
{In the next proposition we prove  estimates for some exit times  which  play a crucial role in obtaining the main result of this section.  Recall that $T_{B_1}$ is the first hitting time of $B_1$.}
 \begin{prop} \label{ball02} Let 
$\psi\in\WLSC{\alpha}{\gamma}$, $\alpha>1$.   Assume that $1>1/\psi^*(1)$. There is $\chi\ge 2$ dependent only on the scalings such that  for $R=  \frac{\chi}{\psi^{-1}(1)}>2$ and  $D_R=(-R,-1) \cup (1,R)$  we have
%
$$\p^x(\tau_{D_R}>1)\ge \Ca \p^x(T_{B_1}>1), \quad 1< |x|\le \frac R2.$$
Moreover
$$\p^0(\tau_{(-R/4,R/4 )}>1)\ge 1/2.$$
The  constant $\Ca$ depends only on the scalings.
\end{prop}

\begin{proof} First observe that $1>1/\psi^*(1)$ is equivalent to $\frac{1}{\psi^{-1}(1)}>1$. Let $x^*$ be the value picked in  Proposition \ref{ball}. We first consider $|x|\ge x^*\ge 2$. By Proposition \ref{ball}, 
$$\p^x(T_{B_1}>2)\ge \Ce \left(\frac{K(|x|)}{K(1/\psi^{-1}(1))}\wedge 1\right).$$
We find $1\le \chi_1\le\chi/4$ satisfying the following conditions
\begin{align*}\frac {K(1/\psi^{-1}(1))} {K(\chi/\psi^{-1}(1))}&\le { \Ce/(2 \Cc)},\\
\frac {V(x^*-1)}{V(\chi/\psi^{-1}(1))}&\le \frac18 (C_3(V(x^*-1)\wedge 1), \\ 
\frac {V(\chi_1/\psi^{-1}(1))}{V(\chi/\psi^{-1}(1))}&\le C_3/(8C_1),\\
\frac{\Cb}{\chi_1^\alpha}&\le \frac12.\end{align*}
Such choice of $\chi_1, \chi$, which are dependent on the scalings,  is possible due to weak lower scaling property  for  $V$, $K$  implied  by  Lemma \ref{ch1V} and Lemma \ref{KWLSC}, respectively.

We set  $R=  \frac{\chi}{\psi^{-1}(1)}$.
The choice of $\chi$ together with Proposition  \ref{MRestimate} imply
$$ \p^x(|X_{\tau_{D_R}}|\ge R) \le  {\Cc}\frac{K(|x|)}{K(R)}\le \frac {\Ce}2\frac{K(|x|-1)}{K(1/\psi^{-1}(1))}.$$
%
Then
 $$\p^x(\tau_{D_R}>2) \ge \p^x(T_{B_1}>2)-  \p^x(|X_{\tau_{D_R}}|\ge R)\ge \frac12  \p^x(T_{B_1}>2)$$
 if $x^*\le |x|\le 1+1/\psi^{-1}(1)$.

 If $1+1/\psi^{-1}(1)<|x|\le 1+  \chi_1/\psi^{-1}(1)$ we use a similar argument based on the exit from a half-space. Indeed, by Lemmas \ref{halfspace} and \ref{exit},
 $$\p^x(\tau_{D_R}>1) \ge \p^x(\tau_{(1,R)}>1)\ge \p^x(\tau_{(1,\infty)}>1)-  \p^x(X_{\tau_{(1,R)}}\ge R)\ge  C_3(V(|x|-1)\wedge1)-\frac {2V(|x|-1)}{V(R-1)}. $$
Observe that, due to Lemma \ref{ch1V},  $C_3(V(|x|-1)\wedge1)\ge C_3(V(1/\psi^{-1}(1))\wedge1) \ge C_3/C_1$ and by the choice of $\chi_1$, $\frac {2V(|x|-1)}{V(R-1)}\le \frac {4V(\chi_1/\psi^{-1}(1))}{V(R)}\le C_3/(2C_1)$. Hence,

$$\p^x(\tau_{D_R}>1) \ge C_3/(2C_1),  \quad 1+1/\psi^{-1}(1)<|x| \le 1+  \chi_1/\psi^{-1}(1).  $$

%
 %

 Next, we assume that $1<|x|<x^*$. By Lemmas \ref{halfspace}, \ref{exit} and the choice of $R$ we have
\begin{eqnarray*}\p^x(\tau_{(1,R)}>1)&\ge& \p^x(\tau_{(1,\infty)}>1)- \p^x(X_{\tau_{(1,R)}}\ge R)\\&\ge&  C_3(V(|x|-1)\wedge1)-\frac {2V(|x|-1)}{V(R-1)} \\
&=&  \frac{V(|x|-1)}{V(x^*-1)} \left[\left(C_3(V(x^*-1)\wedge \frac{V(x^*-1)}{V(x-1)}  \right) - \frac {2V(x^*-1)}{V(R-1)} \right]\\
&\ge&  \frac{V(|x|-1)}{V(x^*-1)} \left[\left(C_3(V(x^*-1)\wedge 1  \right) - \frac {4V(x^*-1)}{V(R)} \right]\\
&\ge&  \frac{V(|x|-1)}{V(x^*-1)} \frac12 C_3(V(x^*-1)\wedge 1)  \\
&\ge&  \frac{V(|x|-1)}{V(x^*-1)} \frac14 C_3\p^{x^*}(\tau_{(1,R )}>1).\end{eqnarray*}
    Moreover, we can  apply Lemma \ref{BHPLowerGeneral} to $F(x)=\E^x\p^{X_{{\tau_{(1,R)}}}}({\tau_{D_R}}>1)$ with $r=x^*-1$. In consequence we can find $c_1$ dependent on the scalings such that
\begin{eqnarray*}\p^x(\tau_{D_R}>1)&\ge& \frac12\left( \p^x(\tau_{(1,R)}>1)+   \E^x\p^{X_{{\tau_{(1,R)}}}}({\tau_{D_R}}>1)\right)\\&\ge& c_1 \frac{V(|x|-1)}{V(x^*-1)} \left( \p^{x^*}(\tau_{(1,R)}>1)+   \E^{x^*}\p^{X_{\tau_{(1,R)}}}(\tau_{D_R}>1)\right)\\&\ge& c_1  \frac{V(|x|-1)}{V(x^*-1)} \p^{x^*}(\tau_{{D_R}}>2)\\
&\ge& \frac {c_1}2  \frac{V(|x|-1)}{V(x^*-1)} \p^{x^*}(T_{B_1}>2),\end{eqnarray*}
where  the last inequality follows from the first part of the proof. Applying Theorem \ref{ball2} we can find $c_2$ dependent only on the scalings such that  
\begin{eqnarray*}\p^x(\tau_{D_R}>1)&\ge& \frac {c_1}2  \frac{V(|x|-1)}{V(x^*-1)} \p^{x^*}(T_{B_1}>2)\ge c_2\p^{x}(T_{B_1}>1),\end{eqnarray*}
which completes the proof for $1<|x|<x^*$.

Finally we consider $1+\chi_1/\psi^{-1}(1)\le |x|\le R/2 $.
Let $R_1= \chi_1/\psi^{-1}(1)\le R/4$. Then by Lemma \ref{ball001}, and the choice of $\chi_1$,
\begin{eqnarray*}\p^x(\tau_{D_R}>1)&\ge&  \p^0(\tau_{(-R_1,R_1) }>1) \ge 1- \frac{\Cb}{\chi_1^\alpha}\ge 1/2 \end{eqnarray*}
and
\begin{eqnarray*}  \p^0(\tau_{(-R/4,R/4) }>1) \ge 1/2 .\end{eqnarray*}
The proof is completed.
\end{proof}

{Below we recall optimal estimates for the transition density  of an unimodal process $X$ if we assume appropriate scaling conditions.} 

\begin{lem}[\cite{MR3165234} ,Corollary 23]\label{density} Assume that $X$ is unimodal. Let  $\psi\in\WLSC{\alpha}{\gamma}\cap \WUSC{\beta}{\rho}$, $1< \alpha \le \beta<2$. Then
$$p_t(x)\approx\psi^{-1}(1/t) \wedge \frac t{|x|V^2(|x|)}\approx  p_t(0)\wedge \frac t{|x|V^2(|x|)}
\approx  \psi^{-1}(1/t) \wedge t\nu(x), \quad \ t>0,\,\,\, x,\,y \in \R.$$
The comparability constants depend on the scalings.

\end{lem}

The following lemma  is instrumental in estimating the heat kernel $p^{D}$.
For the proof see
 \cite[Lemma~2.2]{MR2677618} or \cite[Lemma~2]{MR2722789}.
\begin{lem}\label{lemppu100} 
  Consider disjoint open sets $U_1, U_3\subset D$.
Let $U_2=D\setminus (U_1\cup U_3)$.  If $x\in
  U_1$, $y \in U_3$ and $t>0$, then
	\begin{align*} p_t^{D}( x, y)&\le \p^x(X_{\tau_{U_1}}\in U_2)\sup_{s<t,\, z\in D_2} p_s( z- y)
    +  (t\wedge \E^x \tau_{U_1}) \sup_{u\in U_1,\, z\in U_3}\nu(z-u),
\\
p_t^{D}( x, y)&\ge   t\,\p^x(\tau_{U_1}>t)
    \,\p^y(\tau_{U_3}>t)\inf_{u\in U_1,\, z\in U_3}\nu(z-u).
  \end{align*}
\end{lem}

{Now we are ready to state and prove the main theorem of this section.}
\begin{thm}  \label{HK}
Let $\psi\in\WLSC{\alpha}{\gamma}\cap \WUSC{\beta}{\rho}$, $1<\alpha\le\beta<2$ and the process is unimodal. Let $D= (-\infty,-r)\cup(r, \infty), \ r>0$.  Then

$$p_t^{D}( x, y)\approx   \p^x(\tau_D>t)\p^y(\tau_D>t) p_t(x-y),\quad \ t>0,\,\,\, x,\,y \in D.$$
The comparability constants depend only on the scalings.
\end{thm}
\begin{proof}
We may assume that $|x|<y$. We find the estimates in the case of fixed  $t=2$ or $t=3$ and $r=1$, keeping all arising constants dependent only on the scalings. Then applying the scaling argument we will be able to extend the estimates for the whole range of times and any $r>0$.
We also assume that $t=1> V^2(1)$. The case $1\le  V^2(1)$ can be deduced from a general bound for the killed semigroup obtained in \cite[Corollary 2.4, Theorem 3.3 and the beginning of Section 5]{MR3249349}.
In what follows all comparabilities  hold  with comparability constants which are either depend only on the scalings or they are absolute. The same remark applies to all constants appearing in the proof.  As mentioned above throughout the proof  we fix $D= (-\infty,-1)\cup(1, \infty)$.

We start with the upper bound. {First, we  prove that there is a constant $c_0$ such that 
\begin{equation}\label{UB11}p_1^{D}(x, y)\le c_0 \p^x(\tau_D>2) p_1(x-y).\end{equation}
To this end we consider two cases.}

{\bf Case 1.} Assume that   $V^2(|x-y|)\le16$.  In this case, by Lemma \ref{density}, { $p_1(x-y)\approx p_{1/2}(0)$.} 
Hence  
%
%
{\begin{eqnarray}p_1^{D}(x, y)&=& \int p_{1/2}^{D}( x, u)p_{1/2}^{D}( u, y)du\nonumber\\&\le&  p_{1/2}(0) \int p_{1/2}^{D}( x,u)du \nonumber\\&\le&  c_1 \p^x(\tau_D>1/2) p_1(x-y)\label{UB12}.\end{eqnarray}}

{\bf Case 2.} Assume that  $16<V(|x-y|)$.
If $V(|x|-1)\geq 1$ then by Theorem  \ref{ball2} (or Lemma \ref{halfspace}), $\p^x(\tau_D>2)\approx 1$ and of course \begin{equation}\label{UB13}p_1^{D}(x, y)\leq p_1(x-y)\leq c_2 \p^x(\tau_D>2) p_1(x-y).\end{equation}
Assume that $V(|x|-1)\le 1$, and  $V(y-1)>2$ (the case $V(|x|-1)<V(y-1)<2$ is included  in Case 1). Let  $R: V^2(R-1)=1$. This implies that $|x|\le R$. Also,   since $1>V^2(1)$, we have $R>2$.

We put $U_1=(-R,-1)\cup(1,R)=D_R$ and $U_3=(y-|x-y|/2,y+|x-y|/2)$. We claim that $(y-x)/4\ge R-x$. Indeed, by subaddativity
$$ V((y-x)/4)\ge (1/4)V(y-x)> 4,$$ while
$$ V(|R-x|)\le V(R)+V(|x|)\le V(R-1)+V(|x|-1)+2V(1)\le 4.$$
Hence
\begin{equation}\label{dist}\inf_{u\in U_1,\, z\in U_3} |z-u|=y-(y-x)/2-R\ge (y-x)/4.\end{equation}
By Proposition  \ref{MRestimate},
$$\p^x(\tau_{U_1}<\tau_D)\le   \Cc\frac{V(|x|-1)K(|x|)}{V(x)K(R)},\quad 1<|x|<R.$$
Since $R\ge 2$ then $1=V(R-1)\le V(R)\le 2V(R-1)\le 2$, and by Lemma \ref{ch1V}, $ R\approx \frac 1{\psi^{-1}(1)}$.
Next, by Lemma \ref{KWLSC},  $K(R)\approx \frac{V^2(R)}{R}\approx \psi^{-1}(1) $  implies that  for all  $x: 1<|x|<R$,

$$\p^x(\tau_{U_1}<\tau_D)\le  c_3\min\left\{\frac{V(|x|-1)} {V(|x|)\psi^{-1}(1)} K(|x| ), 1 \right\} \le c_4 \p^x(\tau_D>1),$$
where the last inequality follows from   Theorem \ref{ball2}.
By Proposition  \ref{exittime} and again by Theorem \ref{ball2},
$$1\wedge \E^x\tau_{U_1}\le \Cd\min\left\{\frac{V(|x|-1)} {V(|x|)\psi^{-1}(1)} K(|x| ), 1 \right\}\le c _5\p^x(\tau_D>1). $$
Let $U_2=D\setminus(U_1\cup U_3)$.  By the estimates of  $p_s( z- y)$ (see Lemma \ref{density}),
$$\sup_{s<1,\, z\in U_2} p_s( z- y)= \sup_{s<1} p(s, (x-y)/2) \leq c_6 p_1(x-y).$$
Moreover by \eqref{dist},
$$\sup_{u\in U_1,\, z\in U_3}\nu(z-u)\le \nu((x-y)/4)\approx \nu(x-y)\leq c_7 p_1(x-y).  $$ Then, by Lemma \ref{lemppu100},
  \begin{eqnarray}p_1^{D}(x, y)&\leq& (1\wedge \E^x\tau_{D_1}) \sup_{u\in D_1,\, z\in D_3}\nu(z-u)+ \sup_{s<1,\, z\in D_2} p_s( z- y)\p^x(\tau_{D_1}<\tau_D)\nonumber\\&\le& (c_5c_7+c_4c_6)\p^x(\tau_D>1) p_1(x-y).\label{UB14}\end{eqnarray}
  
{ Since, by  Theorem  \ref{ball2}, $\p^x(\tau_D>1/2)\approx \p^x(\tau_D>1)\approx  \p^x(\tau_D>2)$ we arrive at \eqref{UB11} by combining \eqref{UB12}, \eqref{UB13} and \eqref{UB14}.}
  
{ Finally, by  the semigroup property and by applying the estimate \eqref{UB11}, and from symmetry of the heat kernel, we have
\begin{equation}\label{HKuB}p_2^{D}( x, y)\leq c_8 \p^x(\tau_D>2)\p^y(\tau_D>2) p_2(x-y).\end{equation}}

To get a  general bound for any $t,r>0$  we consider $Y_s = \frac1r X_{st}, \ s\ge 0$.
 For such a process (fixing $t$ and $r$) its characteristic exponent  is  $\psi^r_t(u)= t\psi(u/r)$ so it has the same scaling characteristics as $\psi(u)$. Let $p^Y, p^{D, Y}$ denote the transition densities for the free and killed  process $Y$, respectively.  We have $p_{2t}(x-y)= p^Y_2((x-y)/r)$ and  $p^{rD}_{2t}(x,y)= p^{D,Y}_2(x/r,y/r)$. Moreover
 $\p^{y/r}(\tau^Y_D>2)= \p^{y}(\tau_{rD}>2t)$. Hence applying
 \eqref{HKuB} to $Y$ we obtain
 \begin{eqnarray*}p_{rD}(2t, x, y)&=& p^{D,Y}_2(x/r,y/r)\\    &\leq& c_8 \p^{y/r}(\tau^Y_D>2)\p^{x/r}(\tau^Y_D>2)p^Y_2((x-y)/r)\\    &=& c_8 \p^{y}(\tau_{rD}>2t)\p^{y}(\tau_{rD}>2t)p_{2t}(x-y).\end{eqnarray*}

Next we deal with the lower bound.

\noindent Let  $R=\frac {\chi}{\psi^{-1}(1)}$, where     $\chi$ is the constant from
 Proposition \ref{ball02}. Recall that  \begin{equation}\label{exitball}\p^0(\tau_{(-R/4, R/4)}>1)\ge 1/2.\end{equation}
Also  note that by Lemma \ref{ch1V}, $1/C_1^2\le V^2(\frac 1{\psi^{-1}(1)})\le  V^2(R)\le (1+\chi^2) V^2(\frac 1{\psi^{-1}(1)})  \le (1+\chi^2)C_1^2 $. 
 %
Next, we define for every $z\in D$,    $U_z= D\cap (-R,R)$ and $B_z= (3R, 4R)$ if $|z|<R/2$ or  $U_z= (z-R/4,z+R/4)$ and $B_z= (z+ 2R, z+3R)$, $z>R/2$ or $B_z= (z-3R, z-2R)$, $z<-R/2$, otherwise. Note that $V(|w|-1)\ge V(R)\ge 1/C_1$ for $w\in B_z$. Hence,  by   \cite[Corollary  3.5]{MR3249349} we have
  \begin{eqnarray*}\label{density2}p_1^{D}( u, v)&\ge& c_9  (V(|u|-1)\wedge 1) (V(|v|-1)\wedge 1)  p_1( u-v)\\
 &\ge& c_9C_1^{-2}  p_1( u- v), \quad u\in B_x,\, v\in B_y.\end{eqnarray*}
  Moreover it is easy to check that $|u-v|\le 2((y-x)\vee R)$, hence $p_1(u - v)\ge c_{10} p_1(x-y)$, which follows from Lemma \ref{density}. In consequence we have
\begin{equation}\label{density2}p_1^{D}( u, v)\ge c_{11}    p_1(x-y), \quad u\in B_x,\,  v\in B_y,\end{equation}
where $c_{11}=c_9c_{10}C_1^{-2}$.
By the semigroup property and \eqref{density2},
 \begin{eqnarray}p_3^{D}(x, y)&=& \int p_1^{D}( x, u)p_1^{D}( u, v)p_1^{D}( v, y)dudv\nonumber\\&\ge& c_{11} p_1(x-y) \int_{B_x} p_1^{D}( x,u)du\int_{B_y} p_1^{D}( x, v)dv.\label{density21}\end{eqnarray}
Next for $u \in B^{\prime\prime}_x$ which is an interval with the same center as $B_x$, and has the length   $|B^{\prime\prime}_x|= |B_x|/2$,    we have by Lemma  \ref{lemppu100},  $$p_1^{D}( x,u)\ge \p^x(\tau_{U_x}>1)\p^u(\tau_{B_x}>1)\inf_{u\in U_x,\, z\in B_x}\nu(z-u)\ge \p^x(\tau_{U_x}>1)\p^u(\tau_{B_x}>1)\nu(5R). $$
Next note that, by \eqref{exitball}, $$\p^u(\tau_{B_x}>1)\ge \p^0(\tau_{(-R/4,-R/4)}>1)\ge 1/2$$ and
$$\p^x(\tau_{U_x}>1)\ge \frac{\Ca}{2} \p^x(\tau_{D}>1),$$
which follows from Proposition \ref{ball02} for $|x|\le R/2$ and  from \eqref{exitball} for $|x|> R/2$.
Hence, using $\nu(5R)\approx \frac 1{RV^2(5R)}\approx \frac1R$, we arrive at
$$\int_{B_x} p_1^{D}( x,u)du\ge \frac{\Ca}{4} \p^x(\tau_{D}>1) \nu(R)|B^{\prime\prime}_x|\ge c_{12}  \frac {|B^{\prime\prime}_x|}{R}\p^x(\tau_{D}>1)\ge c_{13} \p^x(\tau_{D}>1),$$
since $|B^{\prime\prime}_x|\approx R$.
 The above estimates combined with \eqref{density21} yield
$$p_{D}(3, x, y)\ge c_{11}c_{13}^2 p_1(x-y) \p^x(\tau_{D}>1)\p^y(\tau_{D}>1).$$
Finally, by observing that  $ p_3( x- y)\approx   p_1(x-y)$ (Lemma \ref{density}) and $\p^x(\tau_{D}>3)\approx  \p^x(\tau_{D}>1)$ (Theorem \ref{ball2}), we arrive at
$$p_{D}(3, x, y)\ge c _{14}\p^x(\tau_{D}>3)p(3, x, y) \p^y(\tau_{D}>3).$$
 Applying the same  scaling argument as used for the upper bound we conclude for any $t,r>0$,
$$p_t^{rD}( x, y)\ge c_{14} \p^x(\tau_{rD}>t)p_t(x- y) \p^y(\tau_{rD}>t).$$
The proof is completed.
\end{proof}
\begin{rem} If we consider the semigroup killed upon hitting $\{0\}$ then with the assumptions of Theorem \ref{HK} we can  obtain the following estimate of its transition density $$p_t^{\{0\}^c}( x, y)\approx   \p^x(T_0>t)\p^y(T_0>t) p_t(x-y),\quad \ t>0,\,\,\, x,\,y \neq 0.$$
This can be proved either  by taking the limit in the estimates from Theorem \ref{HK} if $r\to 0$ or by  proving directly following the steps of the proof above.
\end{rem}

Now, we suggest the following extension of Theorem \ref{HK}.
\begin{rem} Let $X$ be a pure jump process. The assumption of unimodality of $X$ can be removed by assuming certain estimates of the symmetric L\'evy density of $X$.  Suppose that
$$\nu(x)\approx  \frac {f(1/|x|)}{|x|}
,\ x\in \R,$$
where $f:[0,\infty)\to [0,\infty)$ is non-decreasing  and  $f \in \WLSC{\alpha}{\gamma_1}\cap \WUSC{\beta}{\rho_1}$, $1<\alpha<\beta<2$. Then according to \cite[Proposition 28]{MR3165234} the characteristic exponent of $X$, $\psi \approx f$, so 
$\psi \in \WLSC{\alpha}{\gamma}\cap \WUSC{\beta}{\rho}$. Moreover, by the result of Chen, Kim and Kumagai \cite{MR2806700} we have 
$$p_t(x)\approx f^{-1}(1/t) \wedge t\frac {f(1/|x|)}{|x|}\approx  p_t(0)\wedge t\frac {f(1/|x|)}{|x|}
\approx  \psi^{-1}(1/t) \wedge t\nu(x).$$ 
We also note that the renewal function of the ladder process is $ V_X(x)\approx \frac 1{\sqrt{\psi(1/x)}}\approx \frac 1{\sqrt{f(1/x)}}$. Hence we conclude that Lemma \ref{density} holds in this case.  Moreover the density $p_t(x)$ is almost non-increasing for $x\in (0,\infty)$ that is there is symmetric $q_t(x)$ non-increasing on $(0,\infty)$ and a constant $c\ge 1$ such that 
$$c^{-1} q_t(x)  \le p_t(x) \le c q_t(x), \quad t>0, x\in \R.$$
 Therefore we can  repeat, with necesarry slight modifications,  all the steps from the proof of Theorem \ref{HK} and obtain its conclusion in this case. The details are left to interested readers.
\end{rem}

 \bibliographystyle{abbrv}

\end{document}